\documentclass[a4paper,11pt]{amsart}

\providecommand{\MR}{\relax\ifhmode\unskip\space\fi MR }

\providecommand{\href}[2]{#2}

\usepackage{enumerate, booktabs, mathtools}
\usepackage{amsmath, amsfonts, amssymb, amsthm}

\usepackage[T1]{fontenc}
\usepackage[utf8]{inputenc}
\usepackage{mathtools} 
\usepackage{color}

\usepackage[all]{xy}
\usepackage[height=22.5cm, width=15.5cm]{geometry}
\usepackage[active]{srcltx}

\numberwithin{equation}{section}

\theoremstyle{plain}
\newtheorem{thm}{Theorem}[section]
\newtheorem{prop}[thm]{Proposition}
\newtheorem{lem}[thm]{Lemma}

\newtheorem*{Prop3.2}{Proposition 3.2}
\newtheorem*{thm4.2}{Theorem 4.2}

\newtheorem*{thm*}{Theorem}

\theoremstyle{definition}

\theoremstyle{remark}
\newtheorem{rem}[thm]{Remark}
\newtheorem{ex}[thm]{Example}
\newtheorem*{assumption}{Important assumption}
\newtheorem*{cl}{Claim}

\newcommand{\mbb}[1]{\mathbb{#1}}
\newcommand{\wt}[1]{\widetilde{#1}}
\newcommand{\wh}[1]{\widehat{#1}}
\newcommand{\ol}[1]{\overline{#1}}
\newcommand{\lie}[1]{{\mathfrak{#1}}}
\newcommand{\abs}[1]{\lvert #1\rvert}
\newcommand{\Abs}[1]{\bigl\lvert #1\bigr\rvert}
\newcommand{\norm}[1]{\lVert #1\rVert}

\newcommand{\vct}[2]{\begin{pmatrix}#1\\#2\end{pmatrix}}

\renewcommand{\leq}{\leqslant}
\renewcommand{\geq}{\geqslant}

\DeclareMathOperator{\id}{id}

\DeclareMathOperator{\Ad}{Ad}

\DeclareMathOperator{\Hom}{Hom}

\DeclareMathOperator{\Aut}{Aut}

\DeclareMathOperator{\Pic}{Pic}

\DeclareMathOperator{\Tr}{Tr}

\DeclareMathOperator{\rk}{rk}

\DeclareMathOperator{\eps}{\varepsilon}
\DeclareMathOperator{\codim}{codim}

\DeclareMathOperator{\diag}{diag}

\DeclareMathOperator{\kod}{kod}


\title{Schottky groups acting on homogeneous rational manifolds}

\author{Christian Miebach}
\address{Univ.~Littoral C\^ote d'Opale, EA 2797 - LMPA - Laboratoire de 
math\'ematiques pures et appliqu\'ees Joseph Liouville, F-62228 Calais, 
France}
\email{christian.miebach@univ-littoral.fr}

\author{Karl Oeljeklaus}
\address{Aix-Marseille Univ, CNRS, Centrale Marseille, I2M, UMR 7373,
CMI, 39, rue F.~Joli\-ot-Curie, 13453 Marseille  Cedex 13, France}
\email{karl.oeljeklaus@univ-amu.fr}

\thanks{The authors would like to thank A.~T.~Huckleberry and P.~Heinzner for
invitations to the Ruhr-Universit\"at Bochum, Germany, where part of the
work was done. The first author is grateful for the hospitality of the Institut
de Math\'ematiques de Marseille (I2M) as well as for an invitation to the
Fakult\"at f\"ur Mathematik of the Universit\"at Duisburg-Essen by D.~Greb. The
second author is partially supported by the ANR project MNGNK, decision
\#ANR-10-BLAN-0118.}

\begin{document}

\begin{abstract}
We systematically study Schottky group actions on homogeneous rational 
ma\-nifolds and find two new families besides those given by Nori's well-known 
construction. This yields new examples of non-K\"ahler compact complex manifolds 
having free fundamental groups. We then investigate their analytic and geometric 
invariants such as the Kodaira and algebraic dimension, the Picard group and 
the deformation theory, thus extending results due to L\'arusson and to Seade 
and Verjovsky. As a byproduct, we see that the Schottky construction allows to 
recover examples of equivariant compactifications of 
${\rm{SL}}(2,\mbb{C})/\Gamma$ for $\Gamma$ a discrete free loxodromic subgroup 
of ${\rm{SL}}(2,\mbb{C})$, previously obtained by A. Guillot.
\end{abstract}

\maketitle

\section{Introduction}

A classical Schottky group acting on the Riemann sphere $\mbb{P}_1$ is given as
follows. Choose $2r$ open discs $U_1,V_1,\dotsc,U_r,V_r\subset\mbb{P}_1$ having
pairwise disjoint closures as well as $r$ loxodromic automorphisms 
$\gamma_1,\dotsc,\gamma_r$ of $\mbb{P}_1$ satisfying $\gamma_j(U_j)=\mbb{P}_1
\setminus\ol{V_j}$. The group $\Gamma\subset\Aut(\mbb{P}_1)$ generated by
$\gamma_1,\dotsc,\gamma_r$ is a free group of rank $r$ acting freely and
properly on the open subset $\mathcal{U}_\varGamma:=\Gamma\cdot
\mathcal{F}_\varGamma$ where
\begin{equation*}
\mathcal{F}_\varGamma:=\mbb{P}_1\setminus\bigcup_{j=1}^r(U_j\cup V_j).
\end{equation*}
Moreover, the quotient $\mathcal{U}_\varGamma/\Gamma$ is a compact Riemann 
surface of genus $r$. One can relax the notion of a classical Schottky group by
considering $2r$ pairwise disjoint open subsets of $\mbb{P}_1$ that are bounded
by arbitrary Jordan curves instead of circles. In this case Koebe showed that
every compact Riemann surface can be obtained as quotient of an open subset of
$\mbb{P}_1$ by a Schottky group. We refer the reader 
to~\cite[Chapter~1.2.5]{CNS} for an account on the history of Schottky groups.

In~\cite{No} Nori extended the construction of Schottky groups to higher
dimensions in order to obtain compact complex manifolds having free fundamental
group of any rank. Let us recall his construction. Let $z,w\in\mbb{C}^{n+1}$
and consider the smooth function on $\mbb{P}_{2n+1}$ given by 
$\varphi[z:w]=\frac{\norm{w}^2}{\norm{z}^2+\norm{w}^2}$. The fibers
$C_a=\varphi^{-1}(a)$ for $a=0,1$ are isomorphic to $\mbb{P}_n$. For
$0<\eps<\frac{1}{2}$ we have the open neighborhoods $U_{\eps}=\{\varphi<\eps\}$
and $V_{\eps}=\{\varphi>1-\eps\}$ of $C_0$ and $C_1$, respectively. For
$\lambda\in\mbb{C}^*$ define an automorphism of $\mbb{P}_{2n+1}$ by
$g_\lambda[z:w]:=[\lambda^{-1}z:\lambda w]$. A direct calculation shows that
$g_\lambda$ maps $U_{\eps}$ biholomorphically to $\mbb{P}_{2n+1}\setminus 
\ol{V}_{\eps}$ if $\abs{\lambda}^2=\frac{1-\eps}{\eps}>1$. Now let
$f_2,\dotsc,f_r$ be $r\geq2$ automorphisms such that $C_0,C_1, 
f_2(C_0),f_2(C_1),\dotsc,f_r(C_0),f_r(C_1)$ are pairwise disjoint and take 
$\eps>0$ sufficiently small such that $U_{\eps},V_{\eps},f_2(U_{\eps}),
f_2(V_{\eps}),\dotsc,f_r(U_{\eps})$, $f_r(V_{\eps})$ have pairwise disjoint
closures. The automorphisms $f_2,\dotsc,f_r$ exist since $\Aut(\mbb{P}_{2n+1})$ 
acts transitively on the set of disjoint pairs of linearly embedded 
$\mbb{P}_n$'s. Fix $\lambda\in\mbb{C}^*$ with 
$\abs{\lambda}^2=\frac{1-\eps}{\eps}$ and define $r$ automorphisms of 
$\mbb{P}_{2n+1}$ by $\gamma_1:=g_\lambda$ and $\gamma_j:=f_j\circ\gamma_1\circ 
f_j^{-1}$ for $2\leq j\leq r$. The group $\Gamma\subset\Aut(\mbb{P}_{2n+1})$ 
generated by $\gamma_1,\dotsc,\gamma_r$ is an example of a Schottky group acting 
on $\mbb{P}_{2n+1}$. As in the one-dimensional case, there is the analogously 
defined open subset $\mathcal{U}_\varGamma$ on which $\Gamma$ acts freely and 
properly such that the quotient $Q_\varGamma:=\mathcal{U}_\varGamma/\Gamma$ is a 
compact complex manifold. The quotient manifolds $Q_\varGamma$ obtained by 
Nori's construction were studied in a more general framework by L\'arusson 
in~\cite{La}. He showed that, under a technical assumption on the generators of 
the Schottky group $\Gamma$ which guarantees that the $4n$-dimensional Hausdorff 
measure of $\mbb{P}_{2n+1}\setminus\mathcal{U}_\varGamma$ is zero, the manifold 
$Q_\varGamma$ has Kodaira dimension $-\infty$, is rationally connected, and
is not Moishezon. For Schottky groups acting on $\mbb{P}_3$ he proved
furthermore that $Q_\varGamma$ has algebraic dimension zero. In~\cite{SV} Seade
and Verjovsky proved for arbitrary $n$ that $Q_\varGamma$ is diffeomorphic to a
smooth fiber bundle over $\mbb{P}_n$ with fiber the connected sum of $r-1$ 
copies of $S^1\times S^{2n+1}$ and, furthermore, they studied the deformation 
theory of $Q_\varGamma$. For related results in dimension $3$ we refer the 
reader to~\cite{K10} and the references therein.

So far the only known examples of Schottky transformation groups are discrete 
subgroups of the automorphism group of $\mbb{P}_{2n+1}$. Under the hypothesis 
that the $2$-dimensional Hausdorff measure of $\mbb{P}_2\setminus 
\mathcal{U}_\varGamma$ is zero, L\'arusson proved that there do not exist 
Schottky groups acting on $\mbb{P}_2$. In~\cite{Ca} Cano generalized this result 
to  $\mbb{P}_{2n}$.
  
This leads naturally to the main purpose of the present paper, namely the 
construction of Schottky group actions on homogeneous rational manifolds 
different from $\mbb{P}_{2n+1}$.

In order to state the results we have to introduce some terminology. A 
\emph{Schottky pair} in a connected compact complex manifold $X$ is a pair of 
disjoint connected compact complex submanifolds $C_0$ and $C_1$ such that there 
is a holomorphic $\mbb{C}^*$-action on $X$ that is free and proper on 
$X\setminus(C_0\cup C_1)$ and has fixed point set $X^{\mbb{C}^*}=C_0\cup C_1$. 
The first ingredient for the construction of new Schottky groups acting on 
homogeneous rational manifolds is the following observation.

\begin{Prop3.2}
Let $G$ be a connected semisimple complex Lie group, let $Q$ be a parabolic 
subgroup of $G$, and let $G_0$ be a non-compact real form of $G$. If the minimal 
$G_0$-orbit in the homogeneous rational manifold $X=G/Q$ is a real hypersurface, 
then $X$ admits a Schottky pair.
\end{Prop3.2}

Its proof is based on~\cite{Ah} and Matsuki duality. In fact, the Schottky
pairs $(C_0,C_1)$ in $X=G/Q$ given by Proposition~\ref{Prop:AssociatedSchottky}
are the compact orbits of $K=K_0^\mbb{C}$ where $K_0$ is a maximal compact 
subgroup of $G_0$.

This proposition strongly demands to classify all triplets $(G,G_0,Q)$ such that 
the minimal $G_0$-orbit in $X=G/Q$ is a hypersurface. Since we did not find this 
classification, which is of independent interest, in the literature, it is 
carried out in an appendix of this paper. As a consequence, the homogeneous 
rational manifolds admitting Schottky pairs coming from a minimal hypersurface 
orbit are $\mbb{P}_{2n+1}$, the Gra{\ss}mannians ${\rm{Gr}}_n(\mbb{C}^{2n})$, 
the quadrics $Q_{2n}$ and the Gra{\ss}mannians ${\rm{IGr}}_n(\mbb{C}^{2n+1})$ of 
subspaces of $\mbb{C}^{2n+1}$ that are isotropic with respect to a 
non-degenerate quadratic form on $\mbb{C}^{2n+1}$. We proceed to determine all 
the cases in which the Schottky pairs can be moved by automorphisms of $X$ in 
order to actually produce Schottky groups. Our main result is the following

\begin{thm4.2}
Let $G$ be a connected semisimple complex Lie group, let $Q$ be a parabolic 
subgroup of $G$, and let $G_0$ be a non-compact real form of $G$ whose minimal 
orbit is a hypersurface in $X=G/Q$. The Schottky pairs giving rise to Schottky 
group actions on $X$ of arbitrary rank $r$ are precisely the ones on 
$\mbb{P}_{2n+1}$, $Q_{4n+2}$ and ${\rm{IGr}}_n(\mbb{C}^{2n+1})$.
\end{thm4.2}

In addition, we construct Schottky groups acting on $Q_{2n+1}$ and on certain
singular subvarieties of $\mbb{P}_{2n+1}$ which are not directly related to
minimal hypersurface orbits.

Associated with a Schottky group $\Gamma$ acting on $X$ we have the quotient 
manifold $Q_\varGamma$. We prove that the compact complex manifold $Q_\varGamma$ 
is non-K\"ahler, rationally connected, and has Kodaira dimension $\kod 
Q_\varGamma=-\infty$, see Proposition~\ref{Prop:kodaira}. Furthermore, we give a 
criterion for the algebraic dimension $a(Q_\varGamma)$ to be zero 
(cf.~Theorem~\ref{Thm:algdim}) and construct examples of $Q_\varGamma$ 
having strictly positive algebraic dimension, see Examples~\ref{Ex:posalgdim1},
\ref{Ex:posalgdim2} and~\ref{Ex:Posalgdim3}. Their algebraic reduction leads to 
almost-homogeneous compact complex manifolds, namely equivariant 
compactifications of $H/\Gamma$ where $H$ is the Zariski closure of $\Gamma$ in 
$\Aut(X)$. These have been previously unknown, except for 
$H={\rm{SL}}(2,\mbb{C})$. More precisely, equivariant compactifications of
${\rm{SL}}(2,\mbb{C})/\Gamma$ for $\Gamma$ a discrete free loxodromic subgroup
of ${\rm{SL}}(2,\mbb{C})$ (cf.~Example~\ref{Ex:SchottkyCompactifications}) have
been studied in great detail by A.~Guillot in~\cite{G}. These examples show
that the statements of~\cite[Proposition~9.3.12]{CNS}
respectively~\cite[Proposition~3.5]{SV} as well as of~\cite[Theorem~9.3.17]{CNS}
respectively \cite[Theorem~3.10]{SV} cannot be true in general. 

We also determine the Picard group of $Q_\varGamma$
(cf.~Theorem~\ref{Thm:Picard}) and establish the dimension and smoothness of its
Kuranishi space of versal deformations (cf.~Theorem~\ref{Thm:deform}). We note
that several of these results are new even in the case $X=\mbb{P}_{2n+1}$.
Others have been obtained by L\'arusson as well as Seade and Verjovsky under
conditions on the Hausdorff dimension of $X\setminus\mathcal{U}_\varGamma$,
which allowed them to apply extension theorems for holomorphic and meromorphic
functions due to Shiffman and for cohomology classes due to Harvey. Replacing 
Shiffman's and Harvey's techniques by results of Andreotti-Grauert, Scheja and 
Merker-Porten, we are able to remove these assumptions on the Hausdorff 
dimension of $X\setminus\mathcal{U}_\varGamma$.

Let us outline the structure of the paper. In Section~2 we review the basic
facts about Schottky groups in a generality suitable for our purpose. 
Sections~3 and~4 contain the proofs of 
Proposition~\ref{Prop:AssociatedSchottky} and Theorem~\ref{MainTheorem}, 
respectively. In Section~5 we present the technical tools needed to determine 
various cohomology groups of the quotient manifolds $Q_\varGamma$. These are 
then applied in the final Section~6 in order to obtain analytic and geometric 
invariants of $Q_\varGamma$ as well as their deformation theory. The 
classification of the triplets $(G,G_0,Q)$ such that the minimal $G_0$-orbit in 
$X=G/Q$ is a hypersurface is carried out in the appendix.

\section{Complex Schottky groups}

In this section we define Schottky group actions on a connected compact complex 
manifold $X$ in a way that is suitable for the context of this paper.

\subsection{Schottky pairs}\label{Subs:SchottkyPairs}

Let $X$ be a connected compact complex manifold of complex dimension $d$. A
\emph{Schottky pair} in $X$ is given by a pair $(C_0,C_1)$ of disjoint 
connected compact complex submanifolds of $X$ and a holomorphic 
$\mbb{C}^*$-action on $X$ with fixed point set $X^{\mbb{C}^*}=C_0\cup C_1$ that 
is free and proper on $X\setminus(C_0\cup C_1)$. This $\mbb{C}^*$-action 
corresponds to a holomorphic homomorphism $\mbb{C}^*\to\Aut(X)$ denoted by 
$\lambda\mapsto g_\lambda$.

\begin{rem}
Note that the $\mbb{C}^*$-action in Nori's construction described in the
introduction is \emph{not} free. In fact, this action is not even effective and
we must pass to the free and proper action of the quotient group $\mbb{C}^*/
\{\pm1\}$. Under the identification $\mbb{C}^*/\{\pm1\}\cong\mbb{C}^*$ induced 
by $\lambda\mapsto\lambda^2$ the formula for $g_\lambda$ given on page~1 
becomes $g_\lambda[z:w]=[z:\lambda w]$. A direct calculation shows that the 
function $\varphi$ introduced on page~1 satisfies
\begin{equation*}
\varphi[z:\lambda 
w]=\frac{\abs{\lambda}^2\varphi[z:w]}{1+\bigl(\abs{\lambda}^2-1\bigr)\varphi[z: 
w]}.
\end{equation*}
In other words, the map $\varphi\colon\mbb{P}_{2n+1}\to[0,1]$ is 
$S^1$-invariant and $\mbb{R}^{>0}$-equivariant, where $\mbb{R}^{>0}$ acts on 
$[0,1]$ by $t\cdot s:=\frac{t^2s}{1+(t^2-1)s}$.
\end{rem}

In the following remark we will see that it is always possible to define a 
function $\varphi$ having these properties.

\begin{rem}\label{Rem:existenceofvarphi}
Since the $\mbb{C}^*$-action on $\Omega:=X\setminus(C_0\cup C_1)$ is free and
proper, we get the trivial smooth principal $\mbb{R}^{>0}$-bundle 
$\Omega/S^1\to\Omega/\mbb{C}^*$, i.e., differentiably one has 
$\Omega/S^1\simeq(\Omega/\mbb{C}^*)\times\mbb{R}^{>0}$ where we identify 
$\mbb{C}^*/S^1$ with $\mbb{R}^{>0}$ via $\lambda\mapsto\abs{\lambda}$.
Therefore we can define an $S^1$-invariant smooth auxiliary function 
$\varphi\colon\Omega\to(0,1)$ as the composition of the projection onto the
second factor with the bijection $\mbb{R}^{>0}\to(0,1)$,
$s\mapsto\frac{s^2}{1+s^2}$. Note that this bijection is 
$\mbb{R}^{>0}$-equivariant. Since $X^{\mbb{C}^*}=C_0\cup C_1$, we may extend 
$\varphi$ continuously to a function $\varphi\colon X\to[0,1]$ such that
$C_0=\varphi^{-1}(0)$ and $C_1=\varphi^{-1}(1)$. One verifies directly
\begin{equation}\label{Eqn:equivariance}
\varphi\bigl(g_\lambda(x)\bigr)=\frac{\abs{\lambda}^2\varphi(x)}{
1+\bigl(\abs{\lambda}^2-1\bigr)\varphi(x)}=t\cdot\varphi(x)\text{ for 
$t:=\abs{\lambda}\in\mbb{R}^{>0}$}.
\end{equation}
In particular, $\varphi$ is a submersion on $\Omega$.
\end{rem}

\begin{rem}
Later on we will choose the function $\varphi\colon X\to[0,1]$ in a special
way in order to have properties analogous to Nori's construction mentioned in
the introduction. 
\end{rem}

For $0<\eps<\frac{1}{2}$ we set $U_{\eps}:=\{\varphi<\eps\}$. Note that the
family of these open sets forms a neighborhood basis of $C_0$. Similarly, the
open sets $V_{\eps}:=\{\varphi>1-\eps\}$ give a neighborhood basis of $C_1$.

\begin{lem}
Suppose that $X$ admits a Schottky pair. Then
\begin{enumerate}[(a)]
\item the $\mbb{C}^*$-action on $X$ maps fibers of $\varphi$ to fibers of
$\varphi$,
\item for every $x\in X\setminus(C_0\cup C_1)$ we have $\lim_{\lambda\to 0}
g_\lambda(x)\in C_0$ and $\lim_{\lambda\to\infty}g_\lambda(x)\in C_1$, and
\item if $0<\eps<1/2$ and $\abs{\lambda}=\frac{1-\eps}{\eps}$, then
$g_\lambda(U_{\eps})=U_{1-\eps}=X\setminus\ol{V_{\eps}}$.
\end{enumerate}
\end{lem}

\begin{proof}
The first two statements follow directly from the equivariance
condition~\eqref{Eqn:equivariance}.

To show the third one, we calculate as follows. For $a\in\mbb{R}^{\geq0}$ we
have
\begin{equation*}
\frac{\frac{(1-\eps)^2}{\eps^2}a}{1+\left(\frac{(1-\eps)^2}{\eps^2}-1\right)a}=
\frac{(1-\eps)^2a}{\eps^2+(1-2\eps)a},
\end{equation*}
and this quantity is less than $1-\eps$ if and only if $a<\eps$. This shows
$g_\lambda(U_{\eps})\subset U_{1-\eps}$. In order to prove $g_\lambda^{-1}(
U_{1-\eps})\subset U_{\eps}$, let $a\in[0,1-\eps)$ and consider
\begin{equation*}
\frac{\frac{\eps^2}{(1-\eps)^2}a}{1+\left(\frac{\eps^2}{(1-\eps)^2}-1\right)a}=
\frac{\eps^2a}{(1-\eps)^2+(2\eps-1)a}<\frac{\eps^2(1-\eps)}{(1-\eps)^2+(2\eps
-1)(1-\eps)}=\eps,
\end{equation*}
as was to be shown.
\end{proof}

\begin{rem}
Suppose that $X$ admits a Schottky pair $(C_0,C_1)$. Often, there exists in
addition a holomorphic involution $s\colon X\to X$ such that
\begin{enumerate}[(1)]
\item $\varphi\circ s=1-\varphi$ and
\item $s\circ g_\lambda=g_{\lambda^{-1}}\circ s$ for all $\lambda\in
\mbb{C}^*$.
\end{enumerate}
In this case $s(C_0)=C_1$, hence $C_0$ and $C_1$ are biholomorphic. Moreover,
the hypersurface $H:=\{\varphi=1/2\}$ is $s$-stable. Since $s$ yields a
biholomorphism between $U_{1/2}$ and $V_{1/2}$, the hypersurface $H$ must be
Levi-symmetric.
\end{rem}

\subsection{Movable Schottky pairs and Schottky groups}

Let $X$ be a connected compact manifold with $\dim_\mbb{C}X=d$ that admits a
Schottky pair $(C_0,C_1)$. We say that this Schottky pair \emph{can be moved}
or is \emph{movable} if for every integer $r\geq2$ there exist automorphisms
$f_2,\dotsc,f_r$ of $X$ such that $C_0$, $C_1$, $f_2(C_0)$, $f_2(C_1),\dotsc,
f_r(C_0),f_r(C_1)$ are pairwise disjoint.

\begin{rem}
Although our definition of movable Schottky pairs is very general, in praxis we
need to consider compact complex manifolds admitting sufficiently many complex
submanifolds as well as global holomorphic automorphisms. Therefore, a
convenient class of manifolds on which one can expect to find Schottky group
actions is given by homogeneous rational manifolds. From Section~3 on, we will
restrict our attention thus to these.
\end{rem}

\begin{ex}
As shown in the introduction, Nori's construction produces movable Schottky
pairs in $X=\mbb{P}_{2n+1}$.
\end{ex}

\begin{ex}
While $X=\mbb{P}_2$ contains many Schottky pairs, see 
Proposition~\ref{Prop:AssociatedSchottky} and 
Theorem~\ref{Thm:HypersurfaceClassification}, none of them is  movable. To see 
this, suppose on the contrary that $(C_0,C_1)$ is a movable Schottky pair in 
$\mbb{P}_2$. Since any two curves in $\mbb{P}_2$ intersect, $C_0$ and $C_1$ must 
be points. Choose $\eps>0$ sufficiently small so that $U_{\eps}$ is contained in 
a ball. Consequently, $V_{\eps}$ contains a domain biholomorphic to
$\mbb{P}_2\setminus\mbb{B}_2$. But this is impossible since such domains cannot
form a neighborhood basis of a point. We refer the reader to~\cite{Ca} for a
related observation.
\end{ex}

Suppose that $(C_0,C_1)$ is movable and fix $f_1,\dotsc,f_r\in\Aut(X)$ as above
where $f_1:=\id_X$. For all $1\leq j\leq r$ choose $\eps_j\in(0,1/2)$ and
$\lambda_j\in\mbb{C}^*$ with $\abs{\lambda_j}=\frac{1-\eps_j}{\eps_j}>1$. Set
$\gamma_j:=f_j\circ g_{\lambda_j}\circ f_j^{-1}$ and $U_j:=f_j(U_{\eps_j})$ and
$V_j:=f_j(V_{\eps_j})$. We always choose $\eps_j$ sufficiently small such that
the open sets $U_1,\dotsc,U_r,V_1,\dotsc,V_r$  have pairwise disjoint closures.

The group $\Gamma\subset\Aut(X)$ generated by $\gamma_1,\dotsc,\gamma_r$ is
called a \emph{Schottky group} associated with the movable Schottky pair $(C_0,
C_1)$. For such a group $\Gamma$ we define
\begin{equation*}
\mathcal{F}_\varGamma:=X\setminus\bigcup_{j=1}^r(U_j\cup V_j)
\quad\text{and}\quad\mathcal{U}_\varGamma:=\bigcup_{\gamma\in\Gamma}\gamma( 
\mathcal{F}_\varGamma).
\end{equation*}
It is clear that $\mathcal{U}_\varGamma$ is a $\Gamma$-invariant
connected subset of $X$. Since $U_1,\dotsc,U_r,V_1,\dotsc,V_r$  have pairwise
disjoint closures, $\gamma(\mathcal{F}_\varGamma)$ is included in the open set
$X\setminus \mathcal{F}_\varGamma$ for every $\gamma\in\Gamma$. From this we
obtain $\Gamma\cdot\mathcal{F}_\varGamma=\Gamma\cdot
\mathring{\mathcal{F}_\varGamma}$, hence that $\mathcal{U}_\varGamma$ is open.

Moreover, if $X$ is simply-connected and if $\codim C_0,\codim C_1\geq2$, then 
$\mathcal{U}_\varGamma$ is likewise simply-connected. This follows from the fact 
that $\mathcal{U}_\varGamma$ is an increasing union of open subsets which are 
homotopy equivalent to $X\setminus C$ where $C$ is the disjoint union of $N$ 
copies of $C_0\cup C_1$, see Subsection~\ref{Subs:Picard}. If $\codim C_0,\codim 
C_1\geq2$, then each of these open sets is simply-connected, hence the same 
holds for $\mathcal{U}_\varGamma$.

The proof of~\cite[Proposition~9.2.8]{CNS} extends literally to give the 
following.

\begin{prop}\label{Prop:BasicProperties}
The Schottky group $\Gamma$ is the free group generated by $\gamma_1,\dotsc,
\gamma_r$ and acts freely and properly on $\mathcal{U}_\varGamma$. The
connected set $\mathcal{F}_\varGamma$ is a fundamental domain for the
$\Gamma$-action on $\mathcal{U}_\varGamma$. Consequently the quotient
$Q_\varGamma:=\mathcal{U}_\varGamma/\Gamma$ is a connected compact complex
manifold. If $X$ is simply-connected and if $\codim C_j\geq2$ for $j=0,1$, then 
the fundamental group of $Q_\varGamma$ is isomorphic to $\Gamma$.
\end{prop}

\begin{rem}
If we take $r=1$, then we have $\Gamma\simeq\mbb{Z}$ and $\mathcal{U}_\varGamma
=X\setminus(C_0\cup C_1)=\Omega$. In this case $Q_\varGamma$ is a holomorphic
fiber bundle over $\Omega/\mbb{C}^*$ with an elliptic curve as fiber.
\end{rem}

\section{Schottky pairs associated with compact hypersurface orbits}

In this section we prove Proposition~\ref{Prop:AssociatedSchottky} which
provides a general method to construct Schottky pairs in homogeneous rational
manifolds.

\subsection{Nori's construction}\label{Subsection:Nori}

We start by reformulating Nori's construction of Schottky groups in
group-theoretical terms. Recall that on $X=\mbb{P}_{2n+1}$ we have the function
\begin{equation*}
\varphi[z:w]:=\frac{\norm{w}^2}{\norm{z}^2+\norm{w}^2},
\end{equation*}
where $(z,w)\in(\mbb{C}^{n+1}\times\mbb{C}^{n+1})\setminus\{0\}$. The
hypersurface $H=\{\varphi=1/2\}=\bigl\{\norm{z}^2-\norm{w}^2=0\bigr\}$ is an
orbit of the real form $G_0:={\rm{SU}}(n+1,n+1)$ of $G={\rm{SL}}(2n+2,
\mbb{C})$. Note that $X=\{\varphi<1/2\}\cup H\cup\{\varphi>1/2\}$ gives the
decomposition of $X$ into $G_0$-orbits. Let $K$ be the complexification of the
maximal compact subgroup
\begin{equation*}
K_0:=\left\{
\begin{pmatrix}
A&0\\0&B
\end{pmatrix};\ A,B\in{\rm{U}}(n+1), \det(A)\det(B)=1\right\}\simeq
{\rm{S}}\bigl({\rm{U}}(n+1)\times{\rm{U}}(n+1)\bigl)
\end{equation*}
of $G_0$. One sees directly that $K$ has likewise precisely three orbits in
$X$, namely the compact orbits $C_0=\bigl\{[z:0]\bigr\}$ and
$C_1=\bigl\{[0:w]\bigr\}$, and the open orbit $\Omega=K\cdot
H=X\setminus(C_0\cup C_1)$. Moreover, for every $\lambda\in \mbb{C}^*$ the
automorphism $g_\lambda\in\Aut(X)$ belongs to the center of $K$.

\begin{rem}
Note that the symplectic group $\wt{G}:={\rm{Sp}}(n+1,\mbb{C})\subset G$ acts
transitively on $X=\mbb{P}_{2n+1}$, too, see~\cite{On} or~\cite{St}. Moreover,
the automorphism $g_\lambda$ is contained in $\wt{G}$ for any $\lambda\in
\mbb{C}^*$. Since $\wt{G}$ has a Zariski-open orbit in ${\rm{Gr}}_{n+1}
(\mbb{C}^{2n+1})$, we can construct Schottky groups acting on $X$ also inside
the symplectic group. In particular, the Zariski closure of such a Schottky
group is contained in ${\rm{Sp}}(n+1,\mbb{C})$.

If $n+1=2k$, the hypersurface $H$ is an orbit of the real form
$\wt{G}_0:={\rm{Sp}}(k,k)$ of ${\rm{Sp}}(n+1,\mbb{C})$ and
$C_0,C_1\simeq\mbb{P}_{2k-1}$ are orbits of $\wt{K}:=
{\rm{Sp}}(k,\mbb{C})\times{\rm{Sp}}(k,\mbb{C})$. In other words, in this case
we have again a real form having a compact hypersurface orbit.
\end{rem}

This observation leads to a systematic way to construct Schottky group actions 
on homogeneous rational manifolds described in the next subsection.

\subsection{Schottky pairs associated with compact hypersurface orbits}

The following pro\-position allows to associate a Schottky pair with a compact 
hypersurface orbit of a real form $G_0$ of $G$ acting on a homogeneous rational 
manifold $X=G/Q$. Its proof is based on Matsuki duality and on Akhiezer's 
paper~\cite{Ah}.

\begin{prop}\label{Prop:AssociatedSchottky}
Let $G$ be a connected complex semisimple group, let $Q$ be a parabolic subgroup 
of $G$, and let $X=G/Q$ be the corresponding homogeneous rational manifold. Let 
$G_0$ be a non-compact real form of $G$ such that the minimal $G_0$-orbit in $X$ 
is a real hypersurface. Then $X$ admits a Schottky pair.
\end{prop}

Before giving the proof let us review the basic ideas of Matsuki duality. Let 
$G_0$ be a non-compact real form of $G$ and let $K_0$ be a maximal compact 
subgroup of $G_0$. We assume that the groups $G_0$ and $K_0$ are connected. We 
consider the complexification $K:=K_0^\mbb{C}$ as a subgroup of $G$ and call it 
a {\it Matsuki partner} of $G_0$. Matsuki duality provides a bijection between 
the $G_0$-orbits and the $K$-orbits in $X=G/Q$, under which open $G_0$-orbits 
correspond to compact $K$-orbits and compact $G_0$-orbits to open $K$-orbits, 
see e.g.~\cite{BreLo}. This implies in particular that $G_0$ has exactly one 
compact orbit in $X=G/Q$. This compact orbit has minimal dimension among all 
$G_0$-orbits and will be called the \emph{minimal} $G_0$-orbit in $X$.

Suppose from now on that the compact $G_0$-orbit in $X=G/Q$ is a hypersurface.
In the appendix we will determine all triplets $(G,Q,G_0)$ for which this is
the case. First we shall deduce some information about the orbits of $G_0$ and
$K$ in $X$.

\begin{lem}
Suppose that the minimal $G_0$-orbit in $X=G/Q$ is a hypersurface. Then $X$
contains exactly three $G_0$-orbits, the minimal one and two open ones.
Moreover, the generic $K_0$-orbit in $X$ is a hypersurface as well.
\end{lem}

\begin{proof}
Since $X=G/Q$ is simply connected, the complement of the minimal $G_0$-orbit
has exactly two connected components by the Jordan-Brouwer separation theorem.
The first claim follows from the fact that $G_0$ must act transitively on these
connected components. For the second one, it is sufficient to note that $K_0$
acts transitively on the minimal $G_0$-orbit.
\end{proof}

Using Matsuki duality we see that the group $K$ has likewise exactly three
orbits in $X$: two compact ones which lie in the open $G_0$-orbits and one open
orbit that contains the compact $G_0$-orbit. We denote the two compact
$K$-orbits by $C_0$ and $C_1$.

\begin{proof}[Proof of Proposition~\ref{Prop:AssociatedSchottky}]
We only have to prove the existence of a holomorphic $\mbb{C}^*$-action on
$X=G/Q$ that verifies the definition of a Schottky pair. Let $\Omega\simeq
K/K_x$ be the open $K$-orbit in $X$. According to~\cite[Theorem~1]{Ah} its
isotropy group is of the form $K_x=P_\chi$ where $P\subset K$ is a parabolic 
subgroup and $P_\chi$ denotes the kernel of a non-trivial character $\chi\colon 
P\to\mbb{C}^*$ on $P$. In other words, the fibration $\Omega\simeq K/K_x\to K/P$ 
is a $\mbb{C}^*$-principal bundle. Hence, there is a free and proper holomorphic
$\mbb{C}^*$-action on $\Omega$.

It follows from~\cite[Theorem~2]{Ah} that this $\mbb{C}^*$-action extends to
all of $X$ in such a way that the two compact $K$-orbits $C_0$ and $C_1$ are
fixed pointwise.
\end{proof}

\section{Homogeneous rational manifolds admitting
movable Schottky pairs}\label{Section:MovableSchottky}

Let $X=G/Q$ be a homogeneous rational manifold where $G$ is a connected 
semisimple complex Lie group and let $G_0$ be a real form of $G$. In this
section we discuss in detail all the examples of compact hypersurface orbits of
$G_0$ that give rise to movable Schottky pairs. As shown in the appendix, the
only cases where the minimal $G_0$-orbit is a hypersurface in $X=G/Q$ are the
following, see~Theorem~\ref{Thm:HypersurfaceClassification}.

\begin{enumerate}[(1)]
\item $G_0={\rm{SU}}(p,q)$ acting on $X=\mbb{P}_{p+q-1}$;
\item $G_0={\rm{Sp}}(p,q)$ acting on $X=\mbb{P}_{2(p+q)-1}$;
\item $G_0={\rm{SU}}(1,n)$ acting on $X={\rm{Gr}}_k(\mbb{C}^{n+1})$;
\item $G_0={\rm{SO}}^*(2n)$ acting on $X=Q_{2n-2}$;
\item $G_0={\rm{SO}}(1,2n)$ acting on $X={\rm{IGr}}_n(\mbb{C}^{2n+1})$;
\item $G_0={\rm{SO}}(2,2n)$ acting on $X={\rm{IGr}}_{n+1}(\mbb{C}^{2n+2})^0$.
\end{enumerate}

Here ${\rm{IGr}}_k(\mbb{C}^n)$ is the set of all $k$-dimensional subspaces of 
$\mbb{C}^n$ which are isotropic with respect to a non-degenerate symmetric 
bilinear form. The homogeneous rational manifold ${\rm{IGr}}_k(\mbb{C}^{2k})$ 
has two isomorphic connected components, see~\cite[Proposition,p.~735]{GH}.
We denote by ${\rm{IGr}}_k(\mbb{C}^{2k})^0$ one of these components.

\begin{rem}
It is well-known that there exists an ${\rm{SO}}(2n+1,\mbb{C})$-equivariant 
biholomorphism between ${\rm{IGr}}_n(\mbb{C}^{2n+1})$ and ${\rm{IGr}}_{n+1}
(\mbb{C}^{2n+2})^0$. This corresponds to the fact that the automorphism group
of ${\rm{IGr}}_n(\mbb{C}^{2n+1})$ is isomorphic to ${\rm{SO}}(2n+2,\mbb{C})$,
see~\cite{On} and~\cite{St}.
\end{rem}

\begin{assumption}\label{Rem:goodneighborhood}
In all cases in which we obtain a Schottky group action associated with a 
compact hypersurface orbit as described in 
Proposition~\ref{Prop:AssociatedSchottky}, we may and will choose the function
$\varphi$ introduced in Remark~\ref{Rem:existenceofvarphi} to be 
$K_0$-invariant, as it was done in Subsection~\ref{Subsection:Nori} for 
$X=\mbb{P}_{2n+1}$. This important assumption will assure the existence of 
subvarieties of $\mathcal{U}_\varGamma$, and therefore of $Q_{\varGamma}$, 
which are biholomorphic to the Schottky pair varieties $C_0$ and $C_1$. This 
fact can be easily verified in each of the examples discussed in this section 
and will be crucial for several arguments in the proofs of complex analytic 
and geometric properties of the quotient varieties $Q_{\varGamma}$.
\end{assumption}

The main result of this section is

\begin{thm}\label{MainTheorem}
Let $G$ be a connected semisimple complex Lie group, let $Q$ be a parabolic 
subgroup of $G$, and let $G_0$ be a non-compact real form of $G$ whose minimal 
orbit is a hypersurface in $X=G/Q$. The Schottky pairs giving rise to Schottky 
group actions on $X$ of arbitrary rank $r$ are precisely the ones on the 
odd-dimensional projective space $\mbb{P}_{2n+1}$, the quadric $Q_{4n+2}$ and  
the isotropic Gra{\ss}mannian $Z_{n}:={\rm{IGr}}_n(\mbb{C}^{2n+1})$. 
Furthermore, if $(C_{0},C_{1})$ denotes a Schottky pair, then $C_0\simeq C_1$ is 
a linear $\mbb{P}_{n}$ in the case $X=\mbb{P}_{2n+1}$, a linear $\mbb{P}_{2n+1}$ 
in the case $X=Q_{4n+2}$ and an equivariantly embedded copy of 
$Z_{n-1}={\rm{IGr}}_{n-1}(\mbb{C}^{2n-1})$ in the case of $X=Z_{n}$. In each of 
these three cases the automorphism group of $X$ acts transitively on the set of 
Schottky pairs.
\end{thm}

The proof is given by considering separately all of the above six cases.

\subsection{The case of projective space}

The Schottky pairs coming from the first two entries in the above list are only
movable if $p=q$: In both cases we have $C_0\simeq\mbb{P}_{p-1}$ and
$C_1\simeq\mbb{P}_{q-1}$. If $p<q$, then $\dim C_1\geq\frac{1}{2}\dim X$. Hence,
$C_1$ cannot be moved away from itself unless $p=q$ in which case we get back
Nori's construction, see Subsection~\ref{Subsection:Nori}.

It is not hard to see that $G={\rm{SL}}(2n,\mbb{C})$ acts transitively on the 
set of Schottky pairs in $X=\mbb{P}_{2n-1}$, i.e., that the set
\begin{equation*}
\bigl\{(C_0,C_1)\in{\rm{Gr}}_n(\mbb{C}^{2n})\times 
{\rm{Gr}}_n(\mbb{C}^{2n});\ C_0\cap C_1=\{0\}\bigr\}
\end{equation*}
is an ${\rm{SL}}(2n,\mbb{C})$-orbit with respect to the diagonal action on 
${\rm{Gr}}_n(\mbb{C}^{2n})\times{\rm{Gr}}_n(\mbb{C}^{2n})$.

In closing we note that $\codim C_0=\codim C_1=(2n-1)-(n-1)=n$, hence that for 
$n\geq2$ Proposition~\ref{Prop:BasicProperties} implies that 
$\mathcal{U}_\varGamma$ is simply-connected in this case. Of course we knew 
already that this cannot be expected for $n=1$ since no compact Riemann surface 
has a free fundamental group of positive rank.

\subsection{The case of complex Gra{\ss}mannians}

Let us consider the action of $G_0={\rm{SU}}(1,n)$ on $X={\rm{Gr}}_k(
\mbb{C}^{n+1})$ for $1\leq k\leq n$. Here we have $K={\rm{GL}}(n,\mbb{C})$ and
the $K$-action on $X$ is induced from the $K$-representation on {
$\mbb{C}^{n+1}=\mbb{C}e_1\oplus(\{0\}\times\mbb{C}^n)$ where 
$e_1=(1,0,\dotsc,0)$}. The compact $K$-orbits in $X$ are
\begin{align*}
C_0&=\bigl\{V\in X;\
V\subset\{z_1=0\}\bigr\}\simeq{\rm{Gr}}_k(\mbb{C}^n)\quad\text{and}\\
C_1&=\{V\in X;\ e_1\in V\}\simeq{\rm{Gr}}_{k-1}(\mbb{C}^n).
\end{align*}

We claim that $C_0$ can only be moved away by an automorphism of $X$ if $k=n$.
Indeed, suppose that $C_0\cap f(C_0)=\emptyset$ for some $f\in\Aut(X)$. Then
$f(C_0)$ is the set of all $k$-dimensional subspaces of $\mbb{C}^{n+1}$ that
are contained in a fixed hyperplane $H$ of $\mbb{C}^{n+1}$. Since
$\dim\bigl((\{0\}\times\mbb{C}^n)\cap H\bigr)\geq n-1$, the subsets $C_0$ and
$f(C_0)$ cannot be disjoint for $k\leq n-1$. A similar argument shows that
$C_1$ and $f(C_1)$ can only be disjoint for $k=1$. Consequently, this Schottky
pair in $X={\rm{Gr}}_k(\mbb{C}^{n+1})$ is only movable for $k=n=1$. In this
case we obtain Schottky groups acting on $\mbb{P}_1$.

\subsection{Schottky groups acting on 
$Q_{2n-2}$}\label{Subsection:QuadricExample}

Let us consider the symmetric bilinear form $b$ on $\mbb{C}^{2n}$ given by the 
matrix $\left( \begin{smallmatrix}0&I_n\\I_n&0\end{smallmatrix}\right)$ and 
let $G$ be the group of its linear isometries having determinant $1$. Then
$G\simeq{\rm{SO}}(2n,\mbb{C})$ acts transitively on the even-dimensional
quadric $Q_{2n-2}:=\bigl\{[z:w]\in\mbb{P}_{2n-1};\ q(z,w)=0\bigr\}$ where
\begin{equation*}
q(z,w)=\langle z,w\rangle=z_1w_1+\dotsb+z_nw_n
\end{equation*}
is the quadratic form associated with $b$.

Due to Theorem~\ref{Thm:HypersurfaceClassification} the real form $G_0= 
{\rm{SO}}^*(2n)=G\cap{\rm{SU}}(n,n)$ has a compact hypersurface orbit in 
$X=Q_{2n-2}$. One verifies directly that the Lie algebra of $G$ has the form
\begin{equation*}
\lie{g}=\left\{
\begin{pmatrix}
A&B\\
C&-A^t\\
\end{pmatrix}
;\ A\in\mbb{C}^{n\times n}, B,C\in\lie{so}(n,\mbb{C})\right\}
\end{equation*}
and that a Matsuki partner of $G_0$ is given by $K=\left\{\left( 
\begin{smallmatrix}A&0\\0&(A^t)^{-1}\end{smallmatrix}\right);\ 
A\in{\rm{GL}}(n,\mbb{C})\right\}$. The two compact $K$-orbits in $X$ are
\begin{equation*}
C_0=\bigl\{[z:0];\ z\in\mbb{C}^n\bigr\}\simeq\mbb{P}_{n-1}\quad\text{and}\quad
C_1=\bigl\{[0:w];\ w\in\mbb{C}^n\bigr\}\simeq\mbb{P}_{n-1},
\end{equation*}
and they form a Schottky pair. The function $\varphi$ will always be chosen as
\begin{equation*}
\varphi[z:w]:=\frac{\norm{w}^2}{\norm{z}^2+\norm{w}^2}.
\end{equation*}

\begin{rem}
Note that $C_0$ and $C_1$ are the images of maximal isotropic (with respect to 
$b$) subspaces of $\mbb{C}^{2n}$ under the projection onto $\mbb{P}_{2n-1}$.
\end{rem}

We claim that this Schottky pair is movable if and only if $n$ is even. Suppose
first that $n$ is odd. Since $G$ is connected, for every $f\in G$ the 
subvarieties $C_0$ and $f(C_0)$ belong to the same connected component of 
the set of $(n-1)$-planes in $Q_{2n-2}$ which we identify with 
${\rm{IGr}}_n(\mbb{C}^{2n})$. According
to~\cite[Proposition,p.~735]{GH} this implies for their intersection in
$Q_{2n-2}$ that
\begin{equation*}
\dim\bigl(C_0\cap f(C_0)\bigr)\equiv n-1\ (\textrm{mod}\  2)=0.
\end{equation*}
Thus $C_0\cap f(C_0)$ is at least $0$-dimensional, i.e., $C_0$ and $f(C_0)$
cannot be disjoint in $X$.

\begin{rem}
For $n=3$ we have $Q_{4}\simeq{\rm{Gr}}_2(\mbb{C}^4)$ where we have already seen
that the Schottky pairs are not movable.
\end{rem}

Now suppose that $n$ is even. We will show that for a generic choice of $B,C\in
\lie{so}(n,\mbb{C})$ the automorphism
\begin{equation*}
f_{B,C}:=f_B\circ f_C:=
\begin{pmatrix}
I_n&B\\0&I_n
\end{pmatrix}
\begin{pmatrix}
I_n&0\\C&I_n
\end{pmatrix}\in G
\end{equation*}
is such that $C_0$, $C_1$, $f_{B,C}(C_0)$ and $f_{B,C}(C_1)$ are pairwise
disjoint. This follows essentially from the fact that for $n$ even generic
matrices in $\lie{so}(n,\mbb{C})$ are invertible. More precisely, note that for
invertible $B,C\in \lie{so}(n,\mbb{C})$ the subspaces $C_0$, $C_1$ and
$f_B(C_1)$ (resp. $C_0$, $C_1$, $f_C(C_0)$) are pairwise disjoint. Since
$f_{B,C}(C_0)=\bigl\{[(I_n+BC)z:Cz];\ z\in\mbb{P}_n\bigr\}$, the claim follows
once we choose $B$ and $C$ invertible such that $I_n+BC$ is likewise
invertible.

In conclusion, we obtain movable Schottky pairs and therefore Schottky group 
actions only on $X=Q_{4k-2}$. Note that in this case the Schottky pairs are 
given by Schottky pairs in $\mbb{P}_{4k-1}$  lying in $Q_{4k-2}$. Consequently, 
there exist Schottky groups acting on $\mbb{P}_{4k-1}$ that leave $Q_{4k-2}$ 
invariant. This means that the quotient manifolds $Q_\varGamma$ obtained from 
$\mbb{P}_{4k-1}$ contain a hypersurface. Moreover, here we have $\codim 
C_0=\codim C_1=(4k-2)-(2k-1)=2k-1$, which is greater or equal to $2$ precisely 
if $k\geq2$. Hence, due to Proposition~\ref{Prop:BasicProperties} the open set 
$\mathcal{U}_\varGamma\subset Q_{4k-2}$ is simply-connected for all $k\geq2$. 
Note that for $k=1$ this is not the case.

\begin{rem}
L\'arusson already observed the existence of Schottky groups acting on
$\mbb{P}_3$ leaving the quadric $Q_2\simeq\mbb{P}_1\times\mbb{P}_1$ invariant,
see~\cite[Proposition~2.2]{La}.
\end{rem}

In closing we note that $G\simeq{\rm{SO}}(4k,\mbb{C})$ acts transitively on the 
set of Schottky pairs in $X=Q_{4k-2}$, i.e., that the set
\begin{equation*}
\bigl\{(C_0,C_1)\in{\rm{IGr}}_{2k}(\mbb{C}^{4k})^{0}\times 
{\rm{IGr}}_{2k}(\mbb{C}^{4k})^{0};\ C_0\cap C_1=\{0\} \bigr\}
\end{equation*}
is a $G$-orbit with respect to the diagonal action. To prove this, we will show 
that we can map any Schottky pair $(C_0',C_1')$ to $(C_0,C_1)$ by some element 
of $G$ where $C_0=\bigl\{[z:0]\in X;\ z\in\mbb{C}^{2k}\bigr\}$ and $C_1= 
\bigl\{[0:w]\in X;\ w\in\mbb{C}^{2k}\bigr\}$. There exists $g\in G$ with 
$g(C_0')=C_0$. Since $g(C_1')$ is an isotropic subspace of $\mbb{C}^{4k}$ 
complementary to $C_0$, it projects surjectively onto the subspace 
$\{z\in\mbb{C}^{4k};\ z_1=\dotsb=z_{2k}=0\}$. Therefore we find a basis of 
$g(C_1')$ consisting of the vectors
\begin{equation*}
(v_1,e_1),\dotsc,(v_{2k},e_{2k})
\end{equation*}
where $v_i\in\mbb{C}^{2k}$ and $(e_1,\dotsc,e_{2k})$ denotes the standard 
basis of $\mbb{C}^{2k}$. The fact that $g(C_1')$ is isotropic means that the 
matrix $B:=(v_{ij})\in\mbb{C}^{2k\times2k}$ is skew-symmetric. Hence, the 
element
\begin{equation*}
g':=
\begin{pmatrix}
I_{2k}&B\\0&I_{2k}
\end{pmatrix}\in G
\end{equation*}
fixes $C_0$ and maps $C_1$ onto $g(C_1')$, which concludes the argument.

\subsection{Schottky groups acting on isotropic Gra{\ss}mannians}

Let $Z_n={\rm{IGr}}_n(\mbb{C}^{2n+1})$ be the set of $n$-dimensional complex
subspaces of $\mbb{C}^{2n+1}$ that are isotropic with respect to the quadratic
form $q(u,z,w)=u^2+2\langle z,w\rangle$, where $u\in\mbb C,z,w\in\mbb C^{n} $. 
Then $Z_n$ is a homogeneous rational manifold of dimension 
$\dim_\mbb{C}Z_n=\frac{n(n+1)}{2}$. The connected isometry group 
$G\simeq{\rm{SO}}(2n+1,\mbb{C})$ of $q$ acts transitively on $Z_n$. A Matsuki 
partner of $G_0={\rm{SO}}(1,2n)$ is the complex Lie group $K\subset G$ having 
Lie algebra
\begin{equation*}
\lie{k}=\left\{
\begin{pmatrix}
0&0&0\\
0&A&B\\
0&C&-A^t\\
\end{pmatrix}
;\ A\in\mbb{C}^{n\times n}, B,C\in\lie{so}(n,\mbb{C})\right\}\simeq
\lie{so}(2n,\mbb{C}).
\end{equation*}
The group $K\simeq{SO}(2n,\mbb{C})$ has three orbits in $Z_n$: the open one
consists of all isotropic subspaces of $\mbb{C}^{2n+1}$ that are not contained
in $\{0\}\times\mbb{C}^{2n}$, while the set of isotropic $n$-dimensional
complex subspaces of $\{0\}\times\mbb{C}^{2n}$ has two connected components
$C_0$ and $C_1$, both homogeneous under $K$, see \cite[Proposition,
p.~735]{GH}. Remark that $C_0\cup C_1$ is homogeneous under
$\wh{K}\simeq{\rm{O}}(2n,\mbb{C})$. Theorem~\ref{Thm:HypersurfaceClassification} 
and Proposition~\ref{Prop:AssociatedSchottky} show that $(C_{0},C_{1})$ is a 
Schottky pair in $Z_n$. This can also be seen directly as follows.

First we claim that the pair $(C_{0},C_{1})$ is movable. Let $g\in G$ and note
that $g(C_0)$ and $g(C_1)$ are the connected components of the space of
isotropic $n$-dimensional subspaces of $g\bigl(\{0\}\times\mbb{C}^{2n}\bigr)$.
If $C_0$, $C_1$, $g(C_0)$ and $g(C_1)$ are not pairwise disjoint, then there
exists an isotropic $n$-dimensional subspace of
$W_g:=\bigl(\{0\}\times\mbb{C}^{2n}\bigr)\cap g\bigl(\{0\}\times\mbb{C}^{2n}
\bigr)$. However, for a generic choice of $g$ we have $\dim W_g=2n-1$ and
$W_g\cap W_g^\perp=\{0\}$. Therefore the dimension of an isotropic subspace of
$W_g$ is at most $n-1$, which proves the claim.

We consider now the analogous situation in  $\mbb{C}^{2n+2}$ with linear
coordinates $(u,z,w)$ where $u=(u_1,u_2)\in\mbb{C}^2$ and
$z=(z_1,\dotsc,z_n)\in\mbb{C}^n$ and $w=(w_1,\dotsc,w_n)\in\mbb{C}^n$. Let $q$
be the quadratic form  given $q(u,z,w)=u_1^2+u_2^2+2\langle z,w\rangle$. The
Lie algebra of its isometry group $\wh{G}\simeq {\rm{SO}}(2n+2,\mbb{C})$ is 
given by
\begin{equation*}
\wh{\lie{g}}=\left\{
\begin{pmatrix}
D & E & F\\
-F^t & A & B\\
-E^t & C & -A^t\\
\end{pmatrix};\ D\in\lie{so}(2,\mbb{C}),E,F\in\mbb{C}^{2\times n}, A\in
\mbb{C}^{n\times n}, B,C\in\lie{so}(n,\mbb{C})\right\}.
\end{equation*}
Take the $(n+1)$-dimensional isotropic subspace $\wh{V}_0:=\bigl\{(u,iu,z,0);\ 
u\in\mbb{C},z\in\mbb{C}^n\bigr\}$ and set $\wh{Z}_n:=\wh{G}\cdot\wh{V}_0$. Note
that $\wh{Z}_n$ is one of the two connected components of the manifold of
isotropic $(n+1)$-dimensional complex subspaces of $\mbb{C}^{2n+2}$ and
$\dim_\mbb{C}\wh{Z}_n=\frac{n(n+1)}{2}$. 

Let $G$ be the subgroup of $\wh{G}$ having Lie algebra
\begin{equation*}
\lie{g}=\left\{
\begin{pmatrix}
D & E & F\\
-F^t & A & B\\
-E^t & C & -A^t\\
\end{pmatrix}\in\wh{\lie{g}};\ D=0, e_{1j}=f_{1j}=0\text{ for all $1\leq j\leq
n$}\right\}\simeq\lie{so}(2n+1,\mbb{C}).
\end{equation*}
Calculating the dimension of the isotropy group $G_{\wh{V}_0}$, we see that 
$G\cdot\wh{V}_0$ is open in $\wh{Z}_n$. Since $G_{\wh{V}_0}$ is parabolic, it
follows that $G$ acts in fact transitively on $\wh{Z}_n$. It turns out that
$G_{\wh{V}_0}=Q_{\Pi\setminus\{\alpha_n\}}$, which implies that $\wh{Z}_n$ is
$G$-equivariantly isomorphic to the set $Z_n$ of isotropic $n$-dimensional
complex subspaces of $\mbb{C}^{2n+1}$, compare~Remark~\ref{Rem:IsotrGrassmann}.
In other words, the group of holomorphic automorphisms of $Z_n$ is isomorphic
to $\wh{G}\simeq{\rm{SO}}(2n+2,\mbb{C})$ (cf.~\cite{On} or~\cite{St}). 

Thus the manifolds $Z_n$ and $\wh{Z}_n$ are the same. The pair of disjoint
compact $K$-orbits $(C_{0},C_{1})$ constructed above in $Z_n$  for
$K\simeq{\rm{SO}}(2n,\mbb{C})$, can be seen in $\wh{Z}_n$ as the pair of compact
orbits of the subgroup (also called) $K$ of ${\rm{SO}}(2n+2,\mbb{C})$ with Lie
algebra 
\begin{equation*}
\lie{k}=\left\{
\begin{pmatrix}
D & E & F\\
-F^t & A & B\\
-E^t & C & -A^t\\
\end{pmatrix}\in\wh{\lie{g}}\, ;\ D=0, E=F=0\right\}.
\end{equation*}
This pair is movable already under the smaller group $G$ as proved before. In
$\wh{Z}_n$ we can now see explicitly the $\mbb{C}^*$-action which makes
$(C_{0},C_{1})$ a Schottky pair. The one-dimensional complex Lie group has
Lie algebra
\begin{equation*}
\mathcal{Z}_{\wh{\lie{g}}}(\lie{k})=\left\{
\begin{pmatrix}
D&0&0\\
0&0&0\\
0&0&0\\
\end{pmatrix};\ D\in\lie{so}(2,\mbb{C})\right\}\simeq\mbb{C},
\end{equation*}
and it is given (as indicated in the preceding formula) by the centralizer of
$K$ in $\wh{G}$.

In this class of examples we have $\codim C_0=\codim 
C_1=\frac{n(n+1)}{2}-\frac{n(n-1)}{2}=n$. Consequently, $\mathcal{U}_\varGamma 
\subset Z_n$ is simply-connected for $n\geq2$, see 
Proposition~\ref{Prop:BasicProperties}. For $n=1$ we have $Z_1\cong\mbb{P}_1$ 
where $\mathcal{U}_\varGamma$ is not simply-connected.

\begin{rem}\label{Rem:Twistor}
These Schottky pairs are closely related to Schottky groups acting by
conformal automorphisms on the sphere $S^{2n}$ as follows, 
see~\cite[Chapter~10]{CNS}. Consider the following \emph{real} 
${\rm{SO}}(2n+1)$-equivariant fibration
\begin{equation*}
\xymatrix{
Z_n=\widehat Z_n={\rm{SO}}(2n+2)/{\rm{U}}(n+1)\ar[r]^-{\simeq} 
& {\rm{SO}}(2n+1)/{\rm{U}}(n)\ar[d]_{\pi} \\
& {\rm{SO}}(2n+1)/{\rm{SO}}(2n)= S^{2n}.
}
\end{equation*}
This is the so called \emph{twistor fibration} of $Z_n$. The fibers of $\pi$
are complex manifolds isomorphic to $Z_{n-1}$, but the foliation is \emph{not}
holomorphic. The M\"obius group $\textrm{ M\"ob}_{+}(S^{2n})={\rm
Conf}_{+}(S^{2n})$ $\simeq{\rm{SO}}(1,2n+1)$ of conformal orientation-preserving
diffeomorphisms of $S^{2n}$ lifts to a holomorphic action of
${\rm{SO}}(1,2n+1)\subset{\rm{SO}}(2n+2,\mbb{C})\simeq\rm{Aut}(Z_n)$,
see~\cite[p.235--239]{CNS}. This implies that \emph{real} Schottky group actions 
on the manifold $S^{2n}$ induce via lifting by $\pi$ holomorphic Schottky group 
actions on $Z_n$.

Let us be more precise. The natural action of the group $L:=\mbb{R}^{>0}$ as
homotheties $S^{2n}$ has two fix points, $p$ and $q$, say. Lifting with $\pi$
and complexifying the lifted group to $L^{\mathbb C}:=\pi^{*}(L)^{\mathbb
C}\simeq \mathbb C^{*}$, gives us a Schottky pair
\begin{equation}\label{Eqn:SchottkyPairinX_n}
\bigl(C_{0}=\pi^{-1}(p),C_{1}=\pi^{-1}(q)\bigr)
\end{equation}
in $Z_{n}$ with the property that $C_{0}$ and $C_{1}$ are biholomorphic to 
$Z_{n-1}$. This Schottky pair is movable since pairs of points are movable in 
$S^{2n}$ under the action of ${\rm{SO}}(1,2n+1)$. Furthermore, one can take the 
function $\varphi$ to be the pull-back of the standard ${\rm{SO}}(2n)$-invariant 
exhaustion function on $S^{2n}$.
\end{rem}

In closing we show that $\Aut(Z_n)$ acts transitively on the set of Schottky 
pairs in $Z_n$. Since two subgroups of ${\rm{SO}}(2n+2,\mbb{C})$ isomorphic to 
${\rm{SO}}(2n, \mbb{C})$ are conjugate, for two copies $C_{0},C_{1}$ of 
$Z_{n-1}$ in $Z_{n}$ there is $h \in {\rm{SO}}(2n+2,\mbb{C})$ such that 
$h(C_{0}) =C_{1}$. It is then easy to see that the variety of all $Z_{n-1}$'s 
in $Z_{n}$ is isomorphic to the even-dimensional quadric $Q_{2n}= 
{\rm{SO}}(2n+2,\mbb{C})/P$. Now let $(C_{0},C_{1})$ be the Schottky 
pair from~\eqref{Eqn:SchottkyPairinX_n} and denote by $L^\mbb{C}\subset 
\Aut(Z_{n})$ the group isomorphic to $\mathbb C^{*} $ corresponding to it. In 
order to prove transitivity on Schottky pairs, it is sufficient to prove that 
for a Schottky component $C_{1}'$ such that $C_{1}'\cap C_{0}= \emptyset$ there 
is an automorphism $g \in P$ such that $g(C_{1})= C_{1}'$, i.e., $g$ stabilizes 
$C_{0}$ and maps $C_{1}$ to $C_{1}'$. In other words, one has to show that 
$P$ acts transitively on the set of $Z_{n-1}$'s in $Z_{n}$ which are disjoint 
from $C_{0}.$ It is well known that $P$ has an open orbit in $Q_{2n}$. The 
isotropy group in ${\rm{SO}}(1,2n+1)$ of $p \in S^{2n}$ acts transitively on 
$S^{2n}\setminus \{q\}$. This implies directly that $P\cdot C_{1}$ is the open 
$P$-orbit in $Q_{2n}$. Furthermore, since $C_{1}'\cap C_{0}= \emptyset$, for 
every arbitrarily small open neighborhood $U$ of $C_{1}$ in $Z_{n}$ there is an 
element $g_{u} \in L^\mbb{C} \subset P$ such that $g_{u}(C_{1}')\subset U$. 
Thus $g_{u}(C_{1}')$ is in a small open neighborhood of $C_{1}$ contained in 
the open orbit of $P$ in the irreducible cycle space component isomorphic to 
$Q_{2n}$. The claim is proved.

\subsection{Schottky groups acting on $Q_{2n-1}$}

Let us discuss an example of Schottky group actions that are not directly 
related to minimal hypersurface orbits. For $z=(z_1,\dotsc,z_n)$ and 
$w=(w_1,\dotsc,w_n)$ let
\begin{equation*}
X=Q_{2n-1}=\bigl\{[u:z:w]\in\mbb{P}_{2n};\ u^2+2\langle z,w\rangle=0\bigr\}.
\end{equation*}
On $X$ we define $\varphi[u:z:w]=\frac{\norm{w}^2}{\norm{z}^2+\norm{w}^2}$ and
$g_\lambda[u:z:w]:=[u:\lambda^{-1}z:\lambda w]$. Remark that this 
$\mbb{C}^{*}$-action is not free, but since the function $\varphi$ is already 
defined, this fact does not matter. One verifies directly that this yields a 
Schottky pair in $X$. This Schottky pair is movable by an argument similar to 
the one given in Subsection~\ref{Subsection:QuadricExample}. Again, we calculate 
$\codim C_0=\codim C_1=(2n-1)-(n-1)=n$, so that $\mathcal{U}_\varGamma$ is 
simply-connected for $n\geq2$, see Proposition~\ref{Prop:BasicProperties}.

\begin{rem}
The map $p\colon Q_{2n-1}\to\mbb{P}_{2n-1}$ given by $p[u:z:w]=[z:w]$ is a
$2:1$ covering with branch locus $\bigl\{\langle z,w\rangle=0\bigr\}\simeq
Q_{2n-2}$. This covering is equivariant with respect to
$G={\rm{SO}}(2n,\mbb{C})$.
\end{rem}

Let us show that ${\rm{SO}}(2n+1,\mbb{C})$ acts transitively on the set of 
Schottky pairs in $X=Q_{2n-1}$. If $(C_0',C_1')$ is any Schottky pair in $X$, 
then there is $g\in{\rm{SO}}(2n+1,\mbb{C})$ such that $g(C_0')=C_0:=\bigl\{ 
(0,z,0);\ z\in\mbb{C}^n\bigr\}$. Then $g(C_1')$ is an $n$-dimensional isotropic 
subspace of $\mbb{C}^{2n+1}$ complementary to $C_0$. In fact, $g(C_1')$ must be 
complementary to $\bigl\{(u,z,0);\ u\in\mbb{C},z\in\mbb{C}^n\bigr\}$, for 
$(u,z,0)\in g(C_1')$ implies $u^2=0$. Now we may finish the proof in the same 
way as in the case of even-dimensional quadrics.

\section{Extension of cohomology groups and $q$-completeness}

\subsection{Some extension theorems}

In this subsection we collect some technical tools which allow us to study
meromorphic functions, differential forms and cohomology groups of the Schottky
quotients $Q_\varGamma$.

\begin{thm}[Scheja~\cite{Sch}]\label{Thm:Scheja}
Let $X$ be a $d$-dimensional complex manifold, $A\subset X$ a closed analytic
subset of pure dimension $m-1$, and $\mathcal{F}$ a locally free sheaf on $X$.
Then, for every $0\leq k\leq d-m-1$, the restriction map
\begin{equation*}
H^k(X,\mathcal{F})\to H^k(X\setminus A,\mathcal{F})
\end{equation*}
is bijective.
\end{thm}

Following Andreotti and Grauert we say that a complex manifold  $M$ of dimension 
$d$ is $q$-complete if $M$ admits a smooth exhaustion function $\rho$ whose Levi 
form has at least $d-q+1$ strictly positive eigenvalues at each point of $M$. 
Under this convention Stein manifolds are precisely the $1$-complete manifolds.

\begin{thm}[Andreotti,
Grauert~{\cite[Th\'eor\`eme~15]{AnGr}}]\label{Thm:AndreottiGrauert}
Let $\Omega$ be a $q$-complete complex manifold of dimension $d$ with exhaustion 
function $\rho$ and $\mathcal{F}$ be a locally free sheaf on $\Omega$. Then, 
for every $0\leq k\leq d-q-1$, the restriction map
\begin{equation*}
H^k(\Omega,\mathcal{F})\to H^k(\Omega\setminus\{\rho<\eps\},\mathcal{F})
\end{equation*}
is bijective.
\end{thm}

We also need an extension theorem for meromorphic functions and a 
$q$-completeness criterion.

\begin{thm}[Merker, Porten~\cite{MP}]\label{Thm:MerkerPorten}
Let $\Omega$ be a $q$-complete complex manifold of dimension $d$ with $q \leq 
d-1$ and $K\subset\Omega$ a compact subset. Then every meromorphic function 
$f\in\mathcal{M}(\Omega\setminus K)$ extends uniquely as a meromorphic function 
to $\Omega$.
\end{thm}

\begin{prop}[{\cite[Proposition~8]{AnNor}}]\label{Prop:AnNor}
Let $X\subset\mbb{P}_n$ be a projective manifold and let $s_1,\dotsc,s_q$ be
holomorphic sections in the hyperplane line bundle of $\mbb{P}_n$. Then 
$\Omega:=X\setminus\{s_1=\dotsb=s_q=0\}$ is $q$-complete.
\end{prop}

\subsection{Cohomology groups of $Q_\varGamma$}

Our goal here is to determine certain cohomology groups of $Q_\varGamma$ and
for this we prove the following lemma.

\begin{lem}\label{Lem:q-completeness}
Let $X$ be either $\mbb{P}_{2n+1}$ or $Q_{2n}$ or $Q_{2n+1}$ and let
$(C_0,C_1)$ be one of the movable Schottky pairs in $X$ described in
Section~\ref{Section:MovableSchottky}. Then $X\setminus C_0$ is
$(n+1)$-complete. For $n\geq4$, the complex manifold $X_n\setminus C$,
$C:=\pi^{-1}(q)$, $q\in S^{2n}$, see Subsection 4.4, is $(d-3)$-complete with 
$d:=\dim_\mbb{C}X_n$.
\end{lem}

\begin{proof}
We apply Proposition~\ref{Prop:AnNor} to each of the first three cases
separately.

For $C_0=\bigl\{[z:0]\in\mbb{P}_{2n+1};\ z\in\mbb{P}_n\bigr\}$ in
$X=\mbb{P}_{2n+1}$ the criterion of Andreotti and Norguet immediately yields
that $X\setminus C_0$ is $(n+1)$-complete.

If $C_0=\bigl\{[z:0]\in Q_{2n};\ z\in\mbb{P}_n\bigr\}$ in $X=Q_{2n}$, then
$X\setminus C_0$ is again $(n+1)$-complete.

Consider $C_0=\bigl\{[0:0:w]\in Q_{2n+1};\ w\in\mbb{P}_n\bigr\}$ in $X=
Q_{2n+1}=\bigl\{[u:z:w];\ u^2+2\langle z,w\rangle=0\bigr\}$. Then we have $C_0=
\{z_1=\dotsb=z_{n+1}=0\}$, so that $X\setminus C_0$ is $(n+1)$-complete.

In order to prove the claim for $X_4$ we use the spinor embedding
$X_4\hookrightarrow\mbb{P}_{15}$. The image of $X_4$ in $\mbb{P}_{15}$ is the
closure of the set of homogeneous coordinates
\begin{equation*}
[1:x_{12}:\dotsb:x_{45}:y_1:\dotsb:y_5]\in\mbb{P}_{15}
\end{equation*}
where $x_{kl}$, $1\leq k<l\leq5$ are the upper triangular entries of a
skew-symmetric matrix $A\in\mbb{C}^{5\times5}$ and $y_1,\dotsc,y_5$ are the
one-codimensional Pfaffians of $A$.

In~\cite[p.~291]{IM} one finds explicit equations for the image of $X_4$ in
$\mbb{P}_{15}$. Using these, it is not hard to see that the Schottky pair
$(C_0,C_1)$ in $\mbb{P}_{15}$ with
\begin{align*}
C_0&=\{x_{12}=\dotsb=x_{15}=y_2=\dotsb=y_5=0\}\\
C_1&=\{u=x_{23}=\dotsb=x_{45}=y_1=0\}
\end{align*}
intersects $X_4$ in the two connected components of ${\rm{IGr}}_3(\mbb{C}^6)
\subset X_4$. Consequently, we have an explicit formula for Nori's function 
$\varphi\in\mathcal{C}^\infty(\mbb{P}_{15})$ as well as for the tangent space of 
$\{\varphi=1/2\}\cap X_4$ at $p=[1:1:0\dotsb:0]$. This allows us to see by a 
direct calculation that the Levi form of the restriction $\varphi|_{X_4}$ at $p$ 
has four strictly positive eigenvalues. Since the generic fiber of 
$\varphi|_{X_4}$ coincides with the hypersurface orbit of 
$K_0={\rm{SO}}(2)\times{\rm{SO}}(8)$, we conclude that the Levi form of every 
$K_0$-invariant exhaustion function on $X_4\setminus(C_0\cap X_4)$ has at least 
four strictly positive eigenvalues at each point. Hence, $X_4\setminus(C_0\cap 
X_4)$ is $7$-complete, as claimed.

For $n\geq5$, we consider the twistor fibration
\begin{equation*}
\xymatrix{
X_n={\rm{SO}}(2n+1)/{\rm{U}}(n)\ar[d]_{\pi} \\
S^{2n}= {\rm{SO}}(2n+1)/{\rm{SO}}(2n),
}
\end{equation*}
and let $p, q$ the two fixed points of the action of the isotropy group
${\rm{SO}}(2n)$ on $S^{2n}$. Then $S^{2n}\setminus \{q\}$ is conformally
isomorphic to $\mbb R^{n}=\{x=(x_{1},...,x_{n})\}$. We identify the point $p$
with the origin in $\mbb R^{n}$ and define the function $\rho(x):=
\sum x_{i}^{2}$ on $\mbb R^{n}$. The functions $\rho$ and $\tilde \rho :=
\rho \circ \pi$ are invariant under the left action of
${\rm{SO}}(2n)$ on $S^{2n}\setminus \{q\}$ and $X_{n}\setminus C$ and
are exhaustion functions.

It is easy to check that there is a commutative diagram

\begin{equation*}
\xymatrix{
X_4={\rm{SO}}(9)/{\rm{U}}(4)\ar[d]_{\pi|_{X_4}}\ar[r]^-{\iota_1}
& {\rm{SO}}(2n+1)/{\rm{U}}(n)\ar[d]_{\pi} \\
S^{8} ={\rm{SO}}(9)/ {\rm{SO}}(8)\ar[r]^-{\iota_2} &
{\rm{SO}}(2n+1)/{\rm{SO}}(2n)=S^{2n},
}
\end{equation*}
such that $X_4$ is equivariantly and holomorphically embedded in $X_n$ for
$n\geq4$. As we have seen above, the Levi form of the restriction of $\tilde
\rho$ to $X_4\setminus(X_4\cap C_0)$ has four strictly positive eigenvalues
everywhere. Since $\tilde\rho$ is ${\rm{SO}}(2n)$-invariant, the same is true
for $\tilde \rho$ on $X_n$. Therefore $X_n\setminus C$ is $(d-3)$-complete.
\end{proof}

Combining Lemma~\ref{Lem:q-completeness} with the extension theorems of
Andreotti-Grauert and Scheja yields some information about cohomology groups of
the Schottky quotient manifolds $Q_\varGamma$. For the following proposition it
is crucial that the neighborhoods of the Schottky pair $(C_0,C_1)$ are defined
via the $K_0$-invariant function $\varphi$, compare the important assumption.

\begin{prop}\label{Prop:cohomology}
Let $X$ be either $\mbb{P}_{2n+1}$ with $n\geq3$ or $Q_{4n+2}$ with $n\geq2$ or
$Q_{2n+1}$ with $n\geq3$ or $X_n$ with $n\geq4$. Let $(C_0,C_1)$ be a movable
Schottky pair in $X$ and let $\Gamma$ be an associated Schottky group of rank
$r\geq2$. Let $\mathcal{F}$ be a locally free analytic sheaf on $Q_\varGamma$
such that $\pi^*\mathcal{F}$ extends to a locally free sheaf on $X$ with
$H^p(X,\pi^*\mathcal{F})=0$ for $p=1,2$. Then, for $0\leq k\leq2$, we have
isomorphisms
\begin{equation*}
H^k(Q_\varGamma,\mathcal{F})\simeq H^k\bigl(\Gamma,H^0(X,\pi^*\mathcal{F})
\bigr).
\end{equation*}
Moreover, $\mathcal{M}(Q_\varGamma)$ can be identified with the set of
$\Gamma$-invariant rational functions on $X$.
\end{prop}

\begin{proof}
In the first step we will show $H^k(\mathcal{U}_\varGamma,\pi^*\mathcal{F}) 
\simeq H^k(X,\pi^*\mathcal{F})$ for $0\leq k\leq2$. To see this, note first that 
the fundamental domain $\mathcal{F}_\varGamma$ contains the submanifold 
$C=f(C_0)$ for some $f\in\Aut(X)$. This follows from the fact that the 
neighborhoods $U_j$ and $V_j$ are defined by the $K_0$-invariant function 
$\varphi$. Due to Scheja's Theorem~\ref{Thm:Scheja} the restriction map 
$H^k(\mathcal{U}_\varGamma,\pi^*\mathcal{F})\to 
H^k(\mathcal{U}_\varGamma\setminus C,\pi^*\mathcal{F})$ is bijective for $0\leq 
k\leq2$. Since $\mathcal{U}_\varGamma\setminus C$ is a domain in the 
$q$-complete manifold $\Omega=X\setminus C$ with compact complement, an 
application of Andreotti's and Grauert's Theorem~\ref{Thm:AndreottiGrauert} and 
Lemma~\ref{Lem:q-completeness} yields that the restriction map $H^k(X\setminus 
C,\pi^*\mathcal{F})\to H^k(\mathcal{U}_\varGamma\setminus C,\pi^*\mathcal{F})$ 
is also bijective for $0\leq k\leq2$. Another application of Scheja's 
Theorem~\ref{Thm:Scheja} gives the result.

Consequently we get $H^1(\mathcal{U}_\varGamma,\pi^*\mathcal{F})=0=H^2(
\mathcal{U}_\varGamma,\pi^*\mathcal{F})$. This allows us to apply~\cite[Appendix
to \S 2, formula (c)]{Mum} to obtain 
\begin{equation*}
H^k(Q_\varGamma,\mathcal{F})\simeq H^k
\bigl(\Gamma, H^0(\mathcal{U}_\varGamma,\pi^*\mathcal{F})\bigr)= H^k\bigl(
\Gamma, H^0(X,\pi^*\mathcal{F})\bigr)
\end{equation*}
for $k=0,1,2$.
\end{proof}

\begin{rem}
Let $X$ be any homogeneous rational manifold, let $\mathcal{F}$ be either
the structure sheaf $\mathcal{O}$ or the tangent sheaf $\Theta$. Then the
Bott-Borel-Weil theorem shows $H^k(X,\mathcal{F})=0$ for all $k\geq1$.
\end{rem}

\begin{rem}
For $X=\mbb{P}_3$ or $X=Q_3$ our method only yields $H^0(Q_\varGamma,
\mathcal{F})\simeq H^0(X,\pi^*\mathcal{F})^\Gamma$. In addition, for
$X=\mbb{P}_5,Q_5,Q_6$ we have $H^1(Q_\varGamma,\mathcal{F})\simeq
H^1\bigl(\Gamma,H^0(\pi^*\mathcal{F})\bigr)$.
\end{rem}

\section{Geometric properties and deformations of Schottky quotients}

In this section we apply Proposition~\ref{Prop:cohomology} in order to describe
analytic and geometric invariants as well as the deformation theory of Schottky
quotient manifolds. In the whole section $X$ will denote a homogeneous rational
manifold admitting a movable Schottky pair $(C_0,C_1)$ and $\Gamma$ an
associated Schottky group of rank $r\geq2$ with quotient $Q_\varGamma=
\mathcal{U}_\varGamma/\Gamma$.

\subsection{Analytic and geometric invariants}

The following proposition was shown in~\cite{La} for $X=\mbb{P}_n$ under an 
additional assumption on the Hausdorff dimension of $X\setminus 
\mathcal{U}_\varGamma$.

\begin{prop}\label{Prop:kodaira}
The quotient manifold $Q_\varGamma$ is rationally connected and has Kodaira
dimension $-\infty$. If $\codim C_0\geq2$, then $Q_\varGamma$ is not K\"ahler.
\end{prop}

\begin{proof}
The first claim follows again from the fact that we define the open 
neighborhoods $U_j$ and $V_j$ via a $K_0$-invariant function $\varphi\colon
X\to[0,1]$. In this situation we find enough rational curves in the fundamental
domain $\mathcal{F}_\varGamma$ so that we can connect any two points by a chain
of rational curves.

In order to show $\kod(Q_\varGamma)=-\infty$ we apply
Proposition~\ref{Prop:cohomology} to the canonical sheaf
$\mathcal{K}_{Q_\varGamma}$. Then the claim follows from $H^0(X,\mathcal{K}_X^{ 
\otimes m})=0$ for every $m\geq1$ since $X$ is rational.

The last claim is a consequence of the fact that the fundamental group of a 
compact K\"ahler manifold cannot be free. This can be seen as follows,
compare~\cite[Example~1.19]{ABCKT}. Since the rank of the free fundamental
group of a compact complex manifold $Y$ coincides with the first Betti number
of $Y$, we see that no compact K\"ahler manifold can have free fundamental of
odd rank. However, any free group of rank $r$ contains normal subgroups of
any finite index $k$ which are free of rank $k(r-1)+1$ due to the
Nielsen-Schreier theorem. If a compact K\"ahler manifold $Y$ had free
fundamental group of even rank, we could choose a normal subgroup of even index
$k$ and thus would obtain a finite covering of $Y$ having free fundamental
group of odd rank $k(r-1)+1$, a contradiction to the previous observation. This 
proves the last claim since $\mathcal{U}_\varGamma$ is simply-connected if 
$\codim C_0\geq2$.
\end{proof}

Next we give a criterion for the algebraic dimension of $Q_\varGamma$ to be
zero.

\begin{thm}\label{Thm:algdim}
The algebraic dimension $a(Q_\varGamma)$ coincides with the codimension of a
generic $H$-orbit in $X$ where $H$ denotes the Zariski closure of $\Gamma$ in
$\Aut(X)$. In particular, $a(Q_\varGamma)=0$ if and only of $H$ has an open
orbit in $X$.
\end{thm}

\begin{proof}
Due to Theorem~\ref{Thm:MerkerPorten} and Lemma~\ref{Lem:q-completeness} every 
meromorphic function on $Q_\varGamma$ is induced by a $\Gamma$-invariant 
rational function $f$ on $X$. Consequently, $f$ must be invariant under the 
Zariski closure $H$ of $\Gamma$ in $\Aut(X)$. It follows from Rosenlicht's 
theorem~\cite[Theorem~2]{Ro} that the field of $H$-invariant rational functions 
on $X$ has transcendence degree equal to the codimension of a generic $H$-orbit 
in $X$.
\end{proof}

\begin{rem}
It is not difficult to produce examples of Schottky groups $\Gamma$ acting on
$\mbb{P}_{2n+1}$ such that $a(Q_\varGamma)=0$: choose $r=2n+1$ pairwise
disjoint Schottky pairs such that in some point of $\mbb{P}_{2n+1}$ the
corresponding $\mbb{C}^*$-orbits meet transversally.
\end{rem}

It is shown in~\cite[Proposition~2.1]{La} that Nori's Schottky groups acting on
$X=\mbb{P}_3$ yield quotient manifolds of algebraic dimension zero, provided
that the Hausdorff dimension of their limit set is sufficiently small. It can
be deduced from Theorem~\ref{Thm:algdim} that this assumption on
the Hausdorff dimension is superfluous. In~\cite[Proposition~9.3.12]{CNS} and 
\cite[Proposition~3.5]{SV} the same result is claimed to hold for Schottky 
groups acting on $X=\mbb{P}_{2n+1}$. This, however, is not correct as the 
following example shows.

\begin{ex}\label{Ex:posalgdim1}
Let us fix two integers $k\geq1$ and $n\geq2k+1$. Applying Nori's construction
to $X=\mbb{P}(\mbb{C}^{(2k)\times n})\simeq\mbb{P}_{2kn-1}$ gives the Schottky
pair
\begin{align*}
C_0&:=\bigl\{[Z]\in X;\ z_{ij}=0\text{ for all $k+1\leq i\leq 2k$}\bigr\}\simeq
\mbb{P}_{kn-1}\text{ and}\\
C_1&:=\bigl\{[Z]\in X;\ z_{ij}=0\text{ for all $1\leq i\leq k$}\bigr\}\simeq
\mbb{P}_{kn-1}.
\end{align*}
The corresponding $\mbb{C}^*$-action is given by $g_\lambda\in\Aut(X)$,
\begin{equation*}
g_\lambda[Z]=g_\lambda\left[\vct{Z_0}{Z_1}\right]:=\left[\vct{Z_0}{\lambda 
Z_1}\right]
\end{equation*}
where $Z_0,Z_1\in\mbb{C}^{k\times n}$. Let the group $H={\rm{SL}}(2k,\mbb{C})$
act on $X$ by left multiplication. We have $g_\lambda\in H$ for all $\lambda\in
\mbb{C}^*$. Moreover, a direct calculation shows that the Schottky pair
$(C_0,C_1)$ can be moved by elements of $H$. Consequently, there are Schottky
groups $\Gamma$ acting on $X$ with $\Gamma\subset H$.

Due to the First Fundamental Theorem (see e.g.~\cite[p.~387]{Pro}) for $H$ the
invariant ring $\mbb{C}[\mbb{C}^{(2k)\times n}]^H$ is generated by the
$\vct{n}{2k}$ homogeneous $(2k)\times(2k)$ minors of $A$. Since $n\geq 2k+1$,
there are non-constant $H$-invariant rational functions on $X$. Thus the
algebraic dimension of $Q_\varGamma$ is bounded from below by $2k(n-2k)$ for
every Schottky group $\Gamma\subset H$.

Note that $C_0$ and $C_1$ are contained in the $H$-invariant subvariety
$\ol{Y_k}$ where $Y_k:=\bigl\{[Z]\in X;\ \rk(Z)=k\bigr\}$. Therefore we can
view $\ol{Y_k}$ as another projective variety (of dimension $k(k+n)-1$) which
admits actions of Schottky groups. For $k\geq2$ the projective variety
$\ol{Y_k}$ is singular and $C_0$ and $C_1$ meet its singular set $\ol{Y_k}
\setminus Y_k$. In contrast, $Y_1=\ol{Y_1}$ is $H$-equivariantly isomorphic to
$\mbb{P}_1\times\mbb{P}_{n-1}$ where $H={\rm{SL}}(2,\mbb{C})$ acts on
$\mbb{P}_1\times \mbb{P}_{n-1}$ by $h\cdot\bigl([v],[z]\bigr):=
\bigl([hv],[z]\bigr)$. Hence, the Schottky groups acting on $Y_1$ are obtained
as products of Schottky groups acting on $\mbb{P}_1$ and the trivial action on
$\mbb{P}_{n-1}$.
\end{ex}

In the next example we study in detail a fiber of the algebraic reduction map
obtained in the setting of Example~\ref{Ex:posalgdim1}.

\begin{ex}\label{Ex:SchottkyCompactifications}
Take now in the setting of the previous example $n=2k$. In this case the
natural {\it right} action of ${\rm{SL}}(2k,\mbb{C})$ on
$X=\mbb{P}(\mbb{C}^{(2k)\times (2k)})\simeq\mbb{P}_{4k^{2}-1}$ commutes with the
left action and therefore descends to an action on $Q_\varGamma$ which is
constructed as the quotient of the {\it left} action of $\Gamma$. In
particular, the limit set $X\setminus\mathcal{U}_\varGamma$ must be
contained in the complement of the open ${\rm{SL}}(2k,\mbb{C})$-orbit. Hence
$Q_\varGamma$ is an almost homogeneous ${\rm{SL}}(2k,\mbb{C})$-manifold with
open orbit $\Gamma\backslash{\rm{SL}}(2k,\mbb{C})$.

This gives in the case $k=1$ interesting ${\rm{SL}}(2,\mbb{C})$-equivariant
compactifications of $\Gamma\backslash{\rm{SL}}(2,\mbb{C})$ that have previously 
been found by A.~Guillot~\cite{G}. We describe them here 
in a little more detail.

Consider $X=\mbb{P}_{3}\simeq \mbb{P}(\mbb{C}^{2\times 2})$ as an
almost-homogeneous complex manifold under the action of the complex Lie group
${\rm{SL}}(2,\mbb{C}) \times {\rm{SL}}(2,\mbb{C})$, given by left and right
matrix multiplication. Then $X={\rm{PSL}}(2,\mbb{C}) \cup D$, where $D$ is the
$1$-codimensional orbit isomorphic to $\mbb{P}_{1}\times \mbb{P}_{1}$. As above
construct now a (left) Schottky group action on $X$ of $\Gamma \subset
{\rm{SL}}(2,\mbb{C})$. Restricted to $D$ this action is trivial on one
$\mbb{P}_{1}$-factor and just a classical Schottky action on the other.
Therefore the quotient manifold $Q_\varGamma$ is an equivariant holomorphic
compactification of $\Gamma\setminus{\rm{SL}}(2,\mbb{C})$ by a hypersurface $D'
\simeq R \times\mbb{P}_{1}$, where $R$ is the compact Riemann surface associated
to the classical Schottky action of $\Gamma$ on $\mbb{P}_{1}$. By a theorem of
Maskit (\cite{Ma}), it follows that for every finitely generated free, discrete
and loxodromic subgroup $\Gamma \subset {\rm{SL}}(2,\mbb{C})$ the above
construction gives such an equivariant compactification.
\end{ex}

A similar construction is also possible for the case of Schottky group actions 
on quadrics as well as on isotropic Gra{\ss}mannians as the following two 
examples show.

\begin{ex}\label{Ex:posalgdim2}
As in Subsection~\ref{Subsection:QuadricExample} we equip $\mbb{C}^{4k}$ with 
the symmetric bilinear form $b(z,w):=z^tSw$ where $S=\left( 
\begin{smallmatrix}0&I_{2k}\\I_{2k}&0\end{smallmatrix}\right)$. Let $H$ be the 
group of linear isometries of $b$ and note that $H\simeq{\rm{SO}}(4k,\mbb{C})$. 
On $\mbb{C}^{4k\times m}$ for $m\geq1$ we define a
symmetric bilinear form $B$ by the formula $B(Z,W):=\Tr(Z^tSW)$. Then any two 
different columns of $Z\in\mbb{C}^{4k\times m}$ are orthogonal with respect to 
$B$ and, when restricted to one column, $B$ coincides with $b$. The group 
$H\times{\rm{SO}}(m,\mbb{C})$ acts on $\mbb{C}^{4k\times m}$ by left and right 
multiplication and this action leaves $B$ invariant. Consequently, 
$H\times{\rm{SO}}(m,\mbb{C})$ acts on the $(4k-2)$-dimensional quadric
\begin{equation*}
X:=\bigl\{[Z]\in\mbb{P}(\mbb{C}^{4k\times m});\ \Tr(Z^tSZ)=0\bigr\}
\end{equation*}
in $\mbb{P}(\mbb{C}^{4k\times m})\simeq\mbb{P}_{4km-1}$. In contrast with the 
previous example, ${\rm{SO}}(m,\mbb{C})$ does not have an open orbit in $X$.

An argument analogous to the one given in 
Subsection~\ref{Subsection:QuadricExample} shows that the Schottky pair
\begin{align*}
C_0&=\left\{\left[
\begin{pmatrix}
Z_0\\0
\end{pmatrix}
\right]\in X;\ Z_0\in\mbb{C}^{2k\times m}\right\}\simeq\mbb{P}_{2km-1}\\
C_1&=\left\{\left[
\begin{pmatrix}
0\\Z_1
\end{pmatrix}
\right]\in X;\ Z_1\in\mbb{C}^{2k\times m}\right\}\simeq\mbb{P}_{2km-1}
\end{align*}
in $X$ is movable by elements of the group $H$. Therefore there are Schottky 
groups $\Gamma$ acting on $X$ with Zariski closure $\ol{\Gamma}$ contained in 
$H$.

Due to the First Fundamental Theorem for ${\rm{SO}}(4k,\mbb{C})$ the algebra of 
$H$-invariant polynomials on $\mbb{C}^{4k\times m}$ is generated by
\begin{equation*}
p_{ij}(Z)=b(z_i,z_j)\quad\text{for }m\geq1, 1\leq i\leq j\leq m,
\end{equation*}
where $z_1,\dotsc,z_m$ are the columns of $Z\in\mbb{C}^{4k\times m}$, and by
\begin{equation*}
d_{i_1\dotsb i_{4k}}(Z)=\det(z_{i_1}\dotsb z_{i_{4k}})\quad\text{for } m\geq 
4k, 1\leq i_1<i_2<\dotsb<i_{4k}\leq m.
\end{equation*}
Hence, for each $m\geq2$ and every Schottky group $\Gamma\subset H$ the 
quotient manifold $Q_\varGamma$ has positive algebraic dimension.

Concretely, suppose that $k=1$ and $m=2$ and let $\Gamma$ be Zariski dense in 
$H$. Then we have
\begin{equation*}
X=\bigl\{\bigl[(z_1\ z_2)\bigr]\in\mbb{P}(\mbb{C}^{4\times2});\ b(z_1,z_1) 
+ b(z_2,z_2)=0\bigr\}.
\end{equation*}
The algebraic reduction map of $Q_\varGamma$ is induced by the rational mapping
\begin{equation*}
X\dasharrow\mbb{P}_2,[Z]\mapsto\bigl[b(z_1,z_1):b(z_1,z_2):b(z_2,z_2)\bigr],
\end{equation*}
whose image is contained in $\bigl\{[x_0:x_1:x_2]\in\mbb{P}_2;\ x_0=-x_2 
\bigr\}\simeq\mbb{P}_1$. Note that this reduction map is 
${\rm{SO}}(2,\mbb{C})$-equivariant.
\end{ex}

\begin{ex}\label{Ex:Posalgdim3}
Here we construct Schottky groups acting on $Z_{n}={\rm{IGr}}_n 
(\mbb{C}^{2n+1})$ with sufficiently small Zariski closures in
${\rm{SO}}(2n+2,\mbb{C})$ in order to have $\Gamma$-invariant non-constant
meromorphic functions. Then these functions induce meromorphic functions on the
associated quotient manifolds which will be hence of strictly positive algebraic
dimension, compare Theorem~\ref{Thm:algdim}.

To this end we use the twistor fibration, see Remark~\ref{Rem:Twistor}. Let
$M:=S^{m}$ be a round sphere in $S^{2n}$. Its stabilizer is a subgroup of 
${\rm{SO}}(1,2n+1)$ isomorphic to $\textrm{M\"ob}_{+}(S^{m})\simeq
{\rm{SO}}(1,m+1)$. Construct a Schottky group action on $S^{2n}$ by allowing
the pairs of points $(p,q)$ to move only in $M$. Then the lifted Schottky group
has a Zariski closure $H$ contained in a subgroup of ${\rm{SO}}(2n+2,\mbb{C})$
isomorphic to ${\rm{SO}}(m+2,\mbb{C})$. Finally, if $m$ is sufficiently small,
there are $H$-invariant meromorphic functions on $X_{n}$ and the quotient
manifold has strictly positive algebraic dimension.
\end{ex}

\subsection{The Picard group of $Q_\varGamma$}\label{Subs:Picard}

Let $X$ be a homogeneous rational manifold verifying the hypotheses of 
Proposition~\ref{Prop:cohomology}. Applying this proposition to the structure 
sheaf $\mathcal{F}= \mathcal{O}$ we obtain $H^1(Q_\varGamma,\mathcal{O})\simeq 
H^1(\Gamma,\mbb{C})\simeq \Hom(\Gamma_\text{ab},\mbb{C})\simeq\mbb{C}^r$, 
see~\cite[p.~193]{HS}, as well as $H^2(Q_\varGamma,\mathcal{O})\simeq 
H^2(\Gamma,\mbb{C})=0$, see~\cite[Corollary~VI.5.6]{HS}. Hence, the long exact 
cohomology sequence associated with the exponential sequence $0\to\mbb{Z}\to 
\mathcal{O}\to\mathcal{O}^*\to0$ yields
\begin{equation*}
\underset{\simeq \mbb{Z}^r}{H^1(Q_\varGamma,\mbb{Z})}\hookrightarrow 
\underset{\simeq \mbb{C}^r}{H^1(Q_\varGamma,\mathcal{O})}\to 
H^1(Q_\varGamma,\mathcal{O}^*)\to H^2(Q_\varGamma,\mbb{Z})\to0.
\end{equation*}
In order to obtain the Picard group $H^1(Q_\varGamma,\mathcal{O}^*)$ we have 
to determine $H^2(Q_\varGamma,\mbb{Z})$.

\begin{rem}
The subgroup $H^1(Q_\varGamma,\mathcal{O})/H^1(Q_\varGamma,\mbb{Z})\simeq 
(\mbb{C}^*)^r$ of the Picard group of $Q_\varGamma$ consists of the 
topologically trivial line bundles on $Q_\varGamma$, given by representations 
of $\Gamma$ in $\mbb{C}^*$.
\end{rem}

In a first step we will determine $H_2(\mathcal{U}_\varGamma):= 
H_2(\mathcal{U}_\varGamma,\mbb{Z})$. For this, we note that 
$\mathcal{U}_\varGamma$ is the increasing union of the open sets
\begin{equation*}
\Omega_l:=X\setminus
\left(\bigcup_{j=1}^r\bigcup_{\underset{\gamma_{j_l}\not=\gamma_j}
{\gamma\in\Gamma_l}}\gamma(\ol{U}_j)\right)\cup
\left(\bigcup_{j=1}^r\bigcup_{\underset{\gamma_{j_l}\not=\gamma_j^{-1}}
{\gamma\in\Gamma_l}}\gamma(\ol{V}_j)\right)
\end{equation*}
where $\Gamma_l$ denotes the set of all reduced words of length $l\geq1$ in 
$\Gamma$. Since $\Omega_l$ is homotopy equivalent to $X\setminus C$ 
where $C$ is the union of $N_l=2r(2r-1)^{l-1}$ pairwise disjoint copies of 
$C_0$, we have $H_k(\Omega_l)\cong H_k(X\setminus C)$. Let $U$ be a tubular 
neighborhood of $C$ having $N_l$ connected components homeomorphic to 
$C_0\times B$ where $B$ is the unit ball in 
$\mbb{R}^{\dim_\mbb{R}X-\dim_\mbb{R}C_0}$. Then the Mayer-Vietoris sequence of 
the open cover $X=U\cup(X\setminus C)$ with $U\cap(X\setminus C)=U\setminus C$ 
reads
\begin{align*}
\dotsb&\to H_{k+1}(X)\to H_k(U\setminus C)\to H_k(U)\oplus H_k(X\setminus C)\to 
\\ &\to H_k(X)\to H_{k-1}(U\setminus C)\to\dotsb\to H_0(X)\to0.
\end{align*}
Note that $U$ is homotopy equivalent to the disjoint union of $N_l$ copies of 
$C_0$ while $U\setminus C$ is homotopy equivalent to the disjoint union of 
$N_l$ copies of $C_0\times S^{\dim_\mbb{R}X-\dim_\mbb{R}C_0-1}$ where $S^d$ is 
the unit sphere in $\mbb{R}^{d+1}$. Since $C_0$ is homogeneous rational, its 
homology $H_*(C_0)$ is free Abelian. Therefore the K\"unneth formula yields
\begin{equation*}
H_k(C_0\times S^d)\simeq \bigoplus_{j=0}^k H_j(C_0)\otimes H_{k-j}(S^d).
\end{equation*}

Consider the case $k=2$ and suppose that $d\geq3$. Then we have $H_2(C_0\times 
S^d)\simeq H_2(C_0)$. Furthermore, the Mayer-Vietoris sequence starting at 
$H_3(X)=0$ looks like
\begin{equation*}
0\to\underset{\simeq H_2(C_0)^{N_l}}{H_2(U\setminus C)}\to
\underset{\simeq H_2(C_0)^{N_l}}{H_2(U)}\oplus H_2(X\setminus C)\to 
H_2(X)\to H_1(U\setminus C)=0. 
\end{equation*}
From this we obtain $H_2(\Omega_l)\simeq H_2(X\setminus C)\simeq H_2(X)$ for 
all $l\geq1$. Since every singular chain in $\mathcal{U}_\varGamma$ lies in 
$\Omega_l$ for some $l\geq1$, we conclude $H_2(\mathcal{U}_\varGamma) 
\simeq H_2(X)$.

In order to deduce $H_2(Q_\varGamma,\mbb{Z})$, we will use in the second step 
the Cartan-Leray spectral sequence. More precisely, there exists a first 
quadrant spectral sequence of homology type with
\begin{equation*}
E^{p,q}_2\simeq H_p\bigl(\Gamma,H_q(\mathcal{U}_\varGamma)\bigr)
\end{equation*}
and strongly converging to $H_*(Q_\varGamma)$, 
see~\cite[Theorem~8\textsuperscript{bis}.9]{McCl}.

Since $\Gamma$ is free, we have $E^{p,q}_2=0$ for $p\geq2$, 
see~\cite[Corollary~VI.5.6]{HS}. Since the differential $d_2$ is of bidegree 
$(-2,1)$ we obtain
\begin{equation*}
\xymatrix{
{*} & {*} & 0 & & 0\\
H_0\bigl(\Gamma,H_2(X)\bigr) & H_1\bigl(\Gamma,H_2(X)\bigr) & 0\ar[llu] & &
0\ar[lllu]\\
0 & 0 & 0\ar[llu] & & 0\ar[lllu]\\
H_0(\Gamma,\mbb{Z}) & H_1(\Gamma,\mbb{Z}) & 0\ar[llu] & & 0.\ar[lllu]\\
}
\end{equation*}
Consequently, this spectral sequence collapses at $r=2$ and we have 
$E_2^{p,q}=E_{\infty}^{p,q}$ for all $p,q$. In other words, there is an 
increasing filtration $F^*$ on $H_*(Q_\varGamma)$ such that
\begin{align*}
0=E_2^{2,0}&\simeq F^2H_2(Q_\varGamma)/F^1H_2(Q_\varGamma)\\
0=E_2^{1,1}&\simeq F^1H_2(Q_\varGamma)/F^0H_2(Q_\varGamma)\\
E_2^{0,2}&\simeq F^0H_2(Q_\varGamma).
\end{align*}
Since $\Gamma$ is contained in a connected complex Lie group, the induced 
action of $\Gamma$ on $H_q(\mathcal{U}_\varGamma)$ is trivial. Hence, we have 
$E_2^{0,2}\simeq H_0\bigl(\Gamma,H_2(X)\bigr)\simeq H_2(X)$, 
see~\cite[Proposition~VI.3.1]{HS}.

In summary, we have shown $H_2(Q_\varGamma)\simeq H_2(X)$. It follows 
from~\cite[Theorem~3.2.20]{CS} that $H_2(X)\simeq\mbb{Z}$ for every homogeneous 
rational manifold verifying the hypotheses of 
Proposition~\ref{Prop:cohomology}. From this we obtain the following.

\begin{thm}\label{Thm:Picard}
Let $X$ be a homogeneous rational manifold verifying the hypotheses of 
Proposition~\ref{Prop:cohomology}. Let $\Gamma$ be a Schottky group of rank $r$ 
acting on $X$ with associated quotient $\pi\colon\mathcal{U}_\varGamma\to 
Q_\varGamma$. Then the Picard group $H^1(Q_\varGamma,\mathcal{O}^*)$ of 
$Q_\varGamma$ is isomorphic to $(\mbb{C}^*)^r\times\mbb{Z}$.
\end{thm}

\begin{rem}
A generator of $H_2(X)$ and hence of $H_2(Q_\varGamma)$ is given by the 
hyperplane bundle of $X$. Note that this line bundle descends to $Q_\varGamma$ 
since the action of $\Aut(X)^0$ on $\Pic(X)$ is trivial.
\end{rem}

\subsection{Deformation theory of $Q_\varGamma$}

It is possible to embed a Schottky quotient manifold $Q_\varGamma$ into a
complex analytic family in the following way. Fix a movable Schottky pair
$(C_0,C_1)$ in $X$ and automorphisms $f_2,\dotsc,f_r\in\Aut(X)$ such that
$C_0,C_1,f_2(C_0),f_2(C_1),\dotsc,f_r(C_0),f_r(C_1)$ are pairwise disjoint. 
Then, as above, we may choose elements $\lambda_1,\dotsc,\lambda_r\in\mbb{C}^*$ 
with $\abs{\lambda_j}=\frac{1-\eps_j}{\eps_j}>1$. Let $\mathcal{D}\subset
\mbb{C}^r$ be the domain of all possible such $\lambda:=(\lambda_1,\dotsc,
\lambda_r)$. Set $f_1:=\id_X$ and write $\Gamma(\lambda)$ for the Schottky 
group generated by $\gamma_j:=f_j\circ g_{\lambda_j}\circ f_j^{-1}$ for $1\leq 
j\leq r$. Let $F_r=\langle s_1,\dotsc,s_r\rangle$ be the abstract free group of 
rank $r$. We have an action of the free group $F_r$ of rank $r$ on 
$X\times\mathcal{D}$ given by the formula
\begin{equation*}
s_j\cdot(x,\lambda):=\bigl(f_j\circ g_{\lambda_j}\circ 
f_j^{-1}(x),\lambda\bigr).
\end{equation*}
Let $U_j(\lambda)$ and $V_j(\lambda)$ be the open neighborhoods of $f_j(C_0)$ 
and $f_j(C_1)$, respectively, defined for the $\Gamma(\lambda)$-action on $X$.
Then set $\wh{U}_j:=\bigl\{(x,\lambda)\in X\times\mathcal{D};\ x\in 
U_j(\lambda)\bigr\}$ and similarly $\wh{V}_j$ for $1\leq j\leq r$. In the same 
way we define $\wh{\mathcal{F}}$ and $\wh{\mathcal{U}}$ in $X\times\mathcal{D}$.

\begin{prop}
The free group $F_r$ acts freely and properly on $\wh{\mathcal{U}}$ so that we
obtain the commutative diagram
\begin{equation*}
\xymatrix{
\wh{\mathcal{U}}\ar[d]\ar[r] & \wh{\mathcal{U}}/F_r\ar[d]^{\pi}\\
\mathcal{D}\ar[r] & \mathcal{D}.
}
\end{equation*}
The map $\pi\colon\wh{\mathcal{U}}/F_r\to\mathcal{D}$ is a complex analytic 
family in the sense of Kodaira with fibers $Q_{\varGamma(\lambda)}$. In 
particular, all Schottky quotient manifolds are diffeomorphic.
\end{prop}

\begin{proof}
Since we have $s_j(\wh{U}_j)=(X\times\mathcal{D})\setminus \wh{V}_j$ for all 
$1\leq j\leq r$, we may literally copy the proof of the corresponding fact 
without the parameters $\lambda$.
\end{proof}

\begin{prop}
Two Schottky quotient manifolds $Q_\varGamma$ and $Q_{\varGamma'}$ are
biholomorphic if and only if $\Gamma$ and $\Gamma'$ are conjugate in $\Aut(X)$.
\end{prop}

\begin{proof}
Suppose that there exists a biholomorphic map $f\colon Q_\varGamma\to
Q_{\varGamma'}$. Then there exist a biholomorphic map $F\colon
\mathcal{U}_{\varGamma}\to\mathcal{U}_{\varGamma'}$ as well as a group
homomorphism $\varphi\colon\Gamma\to\Gamma'$ such that $F\circ\gamma= \varphi
(\gamma)\circ F$ for all $\gamma\in\Gamma$. Since $X$ is projective, $F$ is
given by finitely many meromorphic functions $f_1,\dotsc,f_N$ on
$\mathcal{U}_\varGamma$. Due to Theorem~\ref{Thm:MerkerPorten} and 
Lemma~\ref{Lem:q-completeness} we may thus extend $F$ as a meromorphic map to 
$X$. It is not hard to show that this extended map is biholomorphic, 
see~\cite{Iva}, hence an element of $\Aut(X)$. Consequently, $\Gamma$ and 
$\Gamma'$ are conjugate in $\Aut(X)$.
\end{proof}

In the rest of this subsection we assume in addition that $X$ verifies the
hypotheses of Proposition~\ref{Prop:cohomology}. In this case,
Proposition~\ref{Prop:cohomology} applied to the tangent sheaf $\Theta$ gives
$H^k(Q_\varGamma,\Theta)\simeq H^k(\Gamma,\lie{g})$ for $0\leq k\leq2$ where
$\Gamma$ acts on $\lie{g}$ via the adjoint representation. Explicitly, we get
$H^0(\Gamma,\lie{g})\simeq\lie{g}^\Gamma$ where $\lie{g}^\Gamma$ denotes the
subspace of $\Gamma$-fixed points. Let $H$ be the Zariski closure of $\Gamma$
in $G$. Then we have $\lie{g}^\Gamma=\lie{g}^H$ which in turn coincides with
the centralizer $\mathcal{Z}_\lie{g}(\lie{h})$ if $H$ is connected. Moreover,
the group of biholomorphic automorphisms $\Aut(Q_{\Gamma})$ is a complex Lie 
group with Lie algebra $\lie{g}^H$.  As noted in~\cite[p.~195]{HS}, 
$H^1(\Gamma,\lie{g})$ is isomorphic to the quotient of $\Hom_\Gamma( I\Gamma, 
\lie{g})$, where $I\Gamma$ denotes the augmentation ideal of $\Gamma$, by the 
submodule of homomorphisms of the form $\varphi_\xi\colon \gamma-e\mapsto 
\Ad(\gamma)\xi-\xi$. Under the identification $\Hom_\Gamma(I\Gamma,\lie{g}) 
\simeq\lie{g}^r$ this submodule corresponds to the image of the map
$\psi\colon\lie{g}\to\lie{g}^r$ given by the formula
\begin{equation*}
\psi(\xi)=\bigl(\Ad(\gamma_1)\xi-\xi,\dotsc,\Ad(\gamma_r)\xi-\xi\bigr),
\end{equation*}
i.e., $H^1(\Gamma,\lie{g})\simeq\lie{g}^r/\psi(\lie{g})$. Note that the kernel 
of $\psi$ is $\lie{g}^\Gamma$. Finally, $H^2(\Gamma,\lie{g})=0$, 
see~\cite[Corollary~VI.5.6]{HS}. In summary, we have the following.

\begin{thm}\label{Thm:deform}
Suppose that $X$ verifies the hypotheses of Proposition~\ref{Prop:cohomology}.
The Kuranishi space of versal deformations of $Q_\varGamma$ is smooth at
$Q_\varGamma$ and of complex dimension $(r-1)\dim\lie{g}+\dim\lie{g}^\Gamma$ 
Moreover, the automorphism group $\Aut(Q_{\Gamma})$ admits as Lie algebra 
$\lie{g}^\Gamma$.
\end{thm}

\begin{rem}
\begin{enumerate}[(a)]
\item If $X$ is either $\mbb{P}_5$ or $Q_5$ or $Q_6$, then we can still compute 
the dimension of the Kuranishi space. However, we do not know whether the 
Kuranishi space is smooth or not.
\item In \cite[Theorem 9.3.17]{CNS} (see also \cite{SV}), the authors claim that 
the Kuranishi space is of dimension  $(r-1)\dim\lie{g}$. But in general the 
above mapping $\psi$ is {\it not } injective, i.e., it is possible that 
$Q_{\Gamma}$ has strictly positive dimensional automorphism group or
is even almost homogeneous, see our examples~\ref{Ex:posalgdim1} 
and~\ref{Ex:SchottkyCompactifications}.
\end{enumerate}
\end{rem}

\begin{appendix}

\section{Minimal orbits of hypersurface type}

Let $G$ be a simply-connected semisimple complex Lie group, let $Q$ be a
parabolic subgroup of $G$, and let $G_0$ be a non-compact simple real form of
$G$. We say that two triplets $(G,G_0,Q)$ and $(G,\wt{G}_0,\wt{Q})$ are
equivalent if there exist $g_1,g_2\in G$ such that $\wt{G}_0=g_1G_0g_1^{-1}$
and $\wt{Q}=g_2Qg_2^{-1}$. In this appendix we outline the classification (up
to equivalence) of all triplets $(G,G_0,Q)$ such that the minimal $G_0$-orbit
is a real hypersurface in $X=G/Q$.

Throughout we write $\sigma\colon\lie{g}\to\lie{g}$ for conjugation with
respect to $\lie{g}_0$. Let $\theta\colon\lie{g}\to\lie{g}$ be a Cartan
involution that commutes with $\sigma$. Then  we have the corresponding Cartan
decomposition $\lie{g}_0=\lie{k}_0\oplus\lie{p}_0$. The analytic subgroup
$K_0$ of $G_0$ having Lie algebra $\lie{k}_0$ is a maximal compact subgroup
of $G_0$.

The following theorem summarizes the outcome of the appendix.

\begin{thm}\label{Thm:HypersurfaceClassification}
Up to equivalence, the homogeneous rational manifolds $X=G/Q$ and the real 
forms $G_0$ having a compact hypersurface orbit in $X$ are the following:
\begin{enumerate}[(1)]
\item $G_0={\rm{SU}}(p,q)$ acting on $X=\mbb{P}_{p+q-1}$;
\item $G_0={\rm{Sp}}(p,q)$ acting on $X=\mbb{P}_{2(p+q)-1}$;
\item $G_0={\rm{SU}}(1,n)$ acting on $X={\rm{Gr}}_k(\mbb{C}^{n+1})$;
\item $G_0={\rm{SO}}^*(2n)$ acting on $X=Q_{2n-2}$;
\item $G_0={\rm{SO}}(1,2n)$ acting on $X={\rm{IGr}}_n(\mbb{C}^{2n+1})$;
\item $G_0={\rm{SO}}(2,2n)$ acting on $X={\rm{IGr}}_{n+1}(\mbb{C}^{2n+2})^0$.
\end{enumerate}
\end{thm}

\subsection{Root-theoretic description of the minimal $G_0$-orbit in $X=G/Q$}

Let $\lie{a}_0$ be a maximal Abelian subspace of $\lie{p}_0$ and let
\begin{equation*}
\lie{g}_0=\lie{m}_0\oplus\lie{a}_0\oplus\bigoplus_{\lambda\in\Lambda}
(\lie{g}_0)_\lambda
\end{equation*}
be the corresponding restricted root space decomposition of $\lie{g}_0$ where
$\lie{m}_0=\mathcal{Z}_{\lie{k}_0}(\lie{a}_0)$ and where $\Lambda=\Lambda(
\lie{g}_0,\lie{a}_0)\subset\lie{a}_0^*\setminus\{0\}$ is the restricted root
system. Choosing a system $\Lambda^+$ of positive restricted roots we obtain
the nilpotent subalgebra $\lie{n}_0:=\bigoplus_{\lambda\in\Lambda^+}
\lie{g}_\lambda$. From this we get the Iwasawa decomposition $G_0=K_0A_0N_0$
where $A_0$ and $N_0$ are the analytic subgroups of $G_0$ having Lie algebras
$\lie{a}_0$ and $\lie{n}_0$, respectively.

Let $\lie{t}_0$ be a maximal torus in $\lie{m}_0$. Then $\lie{h}_0:=\lie{t}_0
\oplus\lie{a}_0$ is a maximally non-compact Cartan subalgebra of $\lie{g}_0$.
Let
\begin{equation*}
\lie{g}=\lie{h}\oplus\bigoplus_{\alpha\in\Delta}\lie{g}_\alpha
\end{equation*}
be the root space decomposition of $\lie{g}$ with respect to the Cartan
subalgebra $\lie{h}:=\lie{h}_0^\mbb{C}$ with root system $\Delta=\Delta(
\lie{g},\lie{h})\subset\lie{h}_\mbb{R}^*\setminus\{0\}$ where $\lie{h}_\mbb{R}
:=i\lie{t}_0\oplus\lie{a}_0$. Let $R\colon\lie{h}_\mbb{R}^*\to\lie{a}_0^*$ be
the restriction operator and let $\Delta_{\sf i}:=\bigl\{\alpha\in\Delta;\
R(\alpha)=0\bigr\}$ be the set of \emph{imaginary} roots.

\begin{rem}
We have $\lie{m}_0^\mbb{C}=\lie{t}_0^\mbb{C}\oplus\bigoplus_{\alpha\in
\Delta_{\sf i}}\lie{g}_\alpha$, i.e., $\Delta_{\sf i}$ is the root system of
$\lie{m}_0^\mbb{C}$ with respect to its Cartan subalgebra $\lie{t}_0^\mbb{C}$.
\end{rem}

Let us define a system $\Delta^+$ of positive roots with respect to the
lexicographic ordering given by a basis of $\lie{h}_\mbb{R}$ whose first $r$
elements form a basis of $\lie{a}_0$. Then, for every $\alpha\in\Delta
\setminus\Delta_{\sf i}$, we have $\alpha\in\Delta^+$ if and only if
$R(\alpha)\in \Lambda^+$ and $R(\Delta^+\setminus\Delta_{\sf i})=\Lambda^+$,
see~\cite[p.156]{Vin3}.

Since the anti-involution $\sigma$ stabilizes $\lie{h}_\mbb{R}$, we obtain an
induced involution on $\lie{h}_\mbb{R}^*$ which we denote again by $\sigma$.
One checks directly that $\sigma$ leaves $\Delta$ invariant and that
$\Delta_{\sf i}=\bigl\{\alpha\in\Delta;\ \sigma(\alpha)=-\alpha\bigr\}$. A root
$\alpha\in\Delta$ is called \emph{real} if $\sigma(\alpha)=\alpha$, and
$\Delta_{\sf r}$ is the set of real roots. Since $R(\alpha)=R\bigl(\sigma
(\alpha)\bigr)$ for all $\alpha\in\Delta$, we get
\begin{equation*}
\sigma(\Delta^+\setminus\Delta_{\sf i})=\Delta^+\setminus\Delta_{\sf i}.
\end{equation*}
In other words, $\Delta^+$ is a $\sigma$-order in the terminology
of~\cite{Ara}.

Before we can state the main result of this subsection, we have to review the
description of parabolic subalgebras of $\lie{g}$ in terms of the root system
$\Delta$. Recall that a root $\alpha\in\Delta^+$ is called \emph{simple} if it
cannot be written as the sum of two positive roots. Let $\Pi\subset\Delta^+$ be
the subset of simple roots. The elements of $\Pi$ form a basis of
$\lie{h}_\mbb{R}^*$ and every positive root can be uniquely written as a linear
combination of simple roots with non-negative integer coefficients.

For an arbitrary subset $\Gamma$ of $\Pi$ we set $\Gamma^{\sf r}:=\langle\Gamma
\rangle_\mbb{Z}\cap\Delta$ and $\Gamma^{\sf n}:=\Delta^+\setminus\Gamma^{\sf
r}$. Then
\begin{equation*}
\lie{q}_\Gamma:=\lie{h}\oplus\bigoplus_{\alpha\in\Gamma^{\sf
r}}\lie{g}_\alpha\oplus
\bigoplus_{\alpha\in\Gamma^{\sf n}}\lie{g}_\alpha
\end{equation*}
is a parabolic subalgebra of $\lie{g}$. The subalgebra $\bigoplus_{\alpha\in
\Gamma^{\sf n}}\lie{g}_\alpha$ is the nilradical of $\lie{q}_\Gamma$, while the
reductive subalgebra $\lie{h}\oplus\bigoplus_{\alpha\in\Gamma^{\sf
r}}\lie{g}_\alpha$
is a Levi subalgebra of $\lie{q}_\Gamma$. Let $Q_\varGamma$ be the analytic
subgroup of $G$ having Lie algebra $\lie{q}_\Gamma$. Then $Q_\varGamma$ is a
parabolic subgroup of $G$ and every parabolic subgroup of $G$ is conjugate to
$Q_\varGamma$ for a suitable choice of $\Gamma\subset\Pi$.

After replacing the triplet $(G,G_0,Q)$ by an equivalent one we may assume that
$G_0\cdot eQ$ is compact in $G/Q$. Due to~\cite[Lemma~3.1]{W} this means that 
there exist a maximally non-compact Cartan subalgebra $\lie{h}_0$ of
$\lie{g}_0$ and a $\sigma$-order $\Delta^+$ of $\Delta=\Delta(\lie{g},\lie{h})$
such that $Q=Q_\varGamma$ for a suitable subset $\Gamma\subset\Pi$.
By~\cite[Theorem~2.12]{W} the real codimension of $G_0\cdot eQ$ in $X$ is given
by $\Abs{\Gamma^{\sf n}\cap\sigma(\Gamma^{\sf n})}$. Therefore the minimal
$G_0$-orbit is a hypersurface if and only if $\Gamma^{\sf n}\cap\sigma
(\Gamma^{\sf n})=\{\alpha_0\}$ for some $\alpha_0\in\Delta^+$.

Suppose that the minimal $G_0$-orbit in $X=G/Q$ is a hypersurface. Then we have
$\sigma(\alpha_0)=\alpha_0$, i.e., $\Delta_{\sf r}^+$ cannot be empty. This
implies that the Lie algebra $\lie{g}$ must be simple, too, and that there are
at least two conjugacy classes of Cartan subalgebras in $\lie{g}_0$.
Furthermore, it is not hard to see that, if $\lie{g}_0$ is a split real form,
then $G_0\simeq{\rm{SL}}(2,\mbb{R})$ and $X\simeq\mbb{P}_1$.

The strategy of the classification is as follows. For every complex simple Lie
algebra $\lie{g}$ and for every real form $\lie{g}_0$ we determine explicitly
the corresponding involution $\sigma$ of $\lie{h}_\mbb{R}^*$ and a
$\sigma$-order $\Delta^+$ of $\Delta=\Delta(\lie{g},\lie{h})$. Then we
enumerate all subsets $\Gamma\subset\Pi\subset\Delta^+$ such that $\Gamma^{\sf
n}\cap\sigma(\Gamma^{\sf n})=\{\alpha_0\}$. This procedure will result in the
list given in the beginning of Section~\ref{Section:MovableSchottky}.

In closing let us note that, if the compact $G_0$-orbit in $X=G/Q$ is a
hypersurface, then $X$ is $K$-spherical. Since the triplets $(G,G_0,Q)$ such
that $X=G/Q$ is $K$-spherical are classified in~\cite[Table~2]{HONO}, the
number of possibilities of $\Gamma$ that have to be checked is further reduced.

The necessary information about root systems and Satake diagrams can be found
in~\cite[Chapter~X.3.3 and Table~VI]{Hel}.

\subsection{The series ${\sf A}_n$}

Let $\lie{g}=\lie{sl}(n+1,\mbb{C})$ with $n\geq1$. The root system of $\lie{g}$
is given by
\begin{equation*}
\Delta=\{\pm(e_k-e_l);\ 1\leq k<l\leq n+1\}
\end{equation*}
where $(e_1,\dotsc,e_{n+1})$ is the standard basis of $\mbb{R}^{n+1}$ and
$\Delta$ is contained in the hypersurface $\{x\in\mbb{R}^{n+1};\ x_1+\dotsb+
x_{n+1}=0\}$. For $\Delta^+:=\{e_k-e_l;\ 1\leq k<l\leq n+1\}$ we have
\begin{equation*}
\Pi=\{\alpha_1=e_1-e_2,\alpha_2=e_2-e_3,\dotsc,\alpha_n=e_n-e_{n+1}\}.
\end{equation*}
A direct calculation shows
\begin{equation*}
e_k-e_l=\alpha_k+\dotsb+\alpha_{l-1}
\end{equation*}
for all $1\leq k<l\leq n+1$.

The non-compact real forms of $\lie{g}$ are $\lie{sl}(n+1,\mbb{R})$, $\lie{sl}
\bigl((n+1)/2,\mbb{H}\bigr)$ if $n+1$ is even, and $\lie{su}(p,q)$ with
$1\leq p\leq q$ and $p+q=n+1$. Since $\lie{sl}(n,\mbb{R})$ is a split real form
and since $\lie{sl}(n,\mbb{H})$ contains only one Cartan subalgebra up to
conjugation, see~\cite[Appendix~C.3]{Kn}, we can restrict attention to
$\lie{g}_0:=\lie{su}(p,q)$. The real rank of $\lie{g}_0$ is
$\rk_\mbb{R}\lie{g}_0:=\dim\lie{a}_0=p$ and the restricted root system
$\Lambda$ is $({\sf BC})_p$ for $p<q$ and ${\sf C}_p$ for $p=q$. The action of
$\sigma$ on $\Delta$ is given by
\begin{equation*}
\sigma(e_k)=
\begin{cases}
-e_{n+2-k} &: 1\leq k\leq p\text{ or }q+1\leq k\leq n+1\\
-e_k &: p+1\leq k\leq q
\end{cases}.
\end{equation*}
This follows from  the fact that $\lie{h}_0=\lie{t}_0\oplus\lie{a}_0$ is
conjugate to the Abelian Lie algebra consisting of all matrices of the form
\begin{equation*}
\diag(it_p+s_1,it_{p-1}+s_2,\dotsc,it_1+s_p,ir_1,\dotsc,ir_{q-p},it_1-s_p,
\dotsc,it_p-s_1)
\end{equation*}
where $t_k,s_l,r_m\in\mbb{R}$ such that $2(t_1+\dotsb+t_p)+r_1+\dotsb+r_{q-p}
=0$. One verifies directly that $\Delta^+$ is a $\sigma$-order.

\begin{rem}
Let $\Gamma_k:=\Pi\setminus\{\alpha_k\}$ for $1\leq k\leq n$. Then we have
$\Gamma_k^{\sf n}=\{e_1-e_{k+1},\dotsc,e_1-e_{n+1},\dotsc,e_k-e_{k+1},\dotsc,
e_k- e_{n+1}\}$. The cardinality of $\Gamma_k^{\sf n}$ is $k(n+1-k)$. The
corresponding homogeneous rational manifold is $X=G/Q_{\varGamma_k}\simeq
{\rm{Gr}}_k(\mbb{C}^{n+1})$.
\end{rem}

\begin{cl}
If the minimal $G_0$-orbit in $X=G/Q$ is a hypersurface, then $Q$ is a maximal
parabolic subgroup of $G$, i.e., $Q=Q_{\varGamma_k}$ for some $1\leq k\leq n$.
\end{cl}

\begin{proof}
Exclude the trivial case $p=q=1$ and suppose that $Q_\varGamma$ is not maximal,
i.e., that $\Gamma\subset\Pi\setminus\{ \alpha_k,\alpha_l\}$ for some $1\leq
k<l\leq n$. Then $\Gamma^{\sf n}$ contains
\begin{multline*}
\Gamma^{\sf n}_k\cup\Gamma^{\sf n}_l=\{ e_1-e_{k+1},\dotsc,e_1-e_{n+1},\dotsc,
e_k-e_{k+1},\dotsc,e_k-e_{n+1},\\e_{k+1}-e_{l+1},\dotsc,e_{k+1}-e_{n+1},
\dotsc,e_{l}-e_{l+1},\dotsc,e_l-e_{n+1}\}.
\end{multline*}

If $p$ is arbitrary and $l=n$, then $\Gamma^{\sf n}$ contains $e_1-e_j$ and
$e_j-e_{n+1}$ for all $k+1\leq j\leq n$. Since we have excluded $p=q=1$, either
we have $1=p<q$ or $2\leq p\leq q$. In the first case $\Gamma^{\sf n}\cap
\sigma(\Gamma^{\sf n})$ contains $e_1-e_{p+1}$ and $e_{p+1}-e_{n+1}$, while in
the second case $\Gamma^{\sf n}\cap\sigma(\Gamma^{\sf n})$ contains $e_1-e_2$
and $e_n-e_{n+1}$. Hence, in both cases the minimal $G_0$-orbit is not a
hypersurface.

If $p=1$ and $1\leq k<l\leq n-1$, then $\Gamma^{\sf n}\cap\sigma(\Gamma^{\sf
n})$ contains again $e_1-e_2$ and $e_2-e_{n+1}$ so that the minimal $G_0$-orbit
is not a hypersurface.

Suppose finally that $p\geq2$ and $1\leq k<l\leq n-1$. Then $\Gamma^{\sf n}\cap
\sigma(\Gamma^{\sf n})$ contains $e_1-e_{n+1}$ and $e_2-e_n$, which finishes
the proof of the claim.
\end{proof}

\begin{cl}
The minimal orbit of $G_0={\rm{SU}}(1,n)$ is a hypersurface in
$X=G/Q_{\varGamma_k}$ for every $1\leq k\leq n$.
\end{cl}

\begin{proof}
Since $p=1$, we have $\sigma(e_j)=-e_j$ for all $2\leq j\leq n$. Therefore, the
only roots in $\Gamma_k^{\sf n}$ which are not imaginary are
$e_1-e_{k+1},\dotsc, e_1-e_{n+1}$ and $e_2-e_{n+1},\dotsc,e_k-e_{n+1}$.
But for $2\leq j\leq n$ only one of the roots $e_1-e_j$ and
$\sigma(e_1-e_j)=e_j-e_{n+1}$ can belong to $\Gamma_k^{\sf n}$, which proves
$\Gamma_k^{\sf n}\cap\sigma(\Gamma_k^{\sf n})=\{e_1-e_{n+1}\}$.
\end{proof}

\begin{cl}
The minimal $G_0$-orbit in $X=G/Q_{\varGamma_1}\simeq\mbb{P}_n$ and
$X=G/Q_{\varGamma_n}\simeq\mbb{P}_n$ is a hypersurface for any
$G_0={\rm{SU}}(p,q)$.
\end{cl}

\begin{proof}
If $\Gamma=\Pi\setminus\{\alpha_1\}$, then $\Gamma^{\sf n}=\{e_1-e_2,\dotsc,
e_1-e_{n+1}\}$. For every $2\leq j\leq n+1$ we have $\sigma(e_1-e_j)=
-\sigma(e_j)-e_{n+1}$ and $\sigma(e_j)=-e_1$ occurs only for $j=n+1$, which
proves $\Gamma^{\sf n}\cap\sigma(\Gamma^{\sf n})=\{e_1-e_{n+1}\}$.

The case $\Gamma=\Pi\setminus\{\alpha_n\}$ can be treated similarly.
\end{proof}

In order to take care of the remaining cases we show

\begin{cl}
Suppose that $p\geq2$ and $2\leq k\leq n-1$. Then the minimal $G_0$-orbit in
$X=G/Q_{\varGamma_k}$ is \emph{not} a hypersurface.
\end{cl}

\begin{proof}
Since $2\leq k\leq n-1$, the set $\Gamma_k^{\sf n}$ contains the two roots
$e_1-e_{n+1}$ and $e_2-e_n$. Moreover, due to $p\geq2$, these roots are real,
hence the claim follows. 
\end{proof}

In summary, we have established the first and third entry in the list given in
the beginning of Section~\ref{Section:MovableSchottky}.

\subsection{The series ${\sf B}_n$}

Let $\lie{g}=\lie{so}(2n+1,\mbb{C})$. The root system of $\lie{g}$ is given by
\begin{equation*}
\Delta=\{\pm e_k;\ 1\leq k\leq n\}\cup\{\pm e_k\pm e_l;\ 1\leq k<l\leq n\}
\end{equation*}
where $(e_1,\dotsc,e_n)$ is the standard basis of $\mbb{R}^n$. For $\Delta^+:=
\{e_k, e_k\pm e_l;\ 1\leq k<l\leq n\}$ we have
\begin{equation*}
\Pi=\{\alpha_1=e_1-e_2, \alpha_2=e_2-e_3,\dotsc,\alpha_{n-1}=e_{n-1}-e_n,
\alpha_n=e_n\}.
\end{equation*}
A direct calculation shows
\begin{align*}
e_k&=\alpha_k+\dotsb+\alpha_n\\
e_k-e_l&=\alpha_k+\dotsb+\alpha_{l-1}\\
e_k+e_l&=\alpha_k+\dotsb+\alpha_{l-1}+2(\alpha_l+\dotsb+\alpha_n)
\end{align*}
for all $1\leq k<l\leq n$.

The only non-compact real forms of $\lie{g}$ are $\lie{g}_0:=\lie{so}(p,q)$ with
$1\leq p\leq q$ and $p+q=2n+1$. The real rank of $\lie{g}_0$ is
$\rk_\mbb{R}\lie{g}_0=p$ and the restricted root system $\Lambda$ coincides with
${\sf B}_p$. The action of $\sigma$ on $\Delta$ is given by
\begin{equation*}
\sigma(e_k)=
\begin{cases}
e_k &: 1\leq k\leq p\\
-e_k &: p+1\leq k\leq p+\left[\frac{q-p}{2}\right]=n
\end{cases}.
\end{equation*}
Therefore the simple roots $\alpha_1,\dotsc,\alpha_{p-1}$ are real and
$\alpha_{p+1},\dotsc,\alpha_n$ are imaginary, while $\sigma(\alpha_p)=e_p+
e_{p+1}$.

According to~\cite[Table~2]{HONO} the only $\Gamma\subset\Pi$ such that the
minimal $G_0$-orbit in $X=G/Q_\Gamma$ might be a hypersurface are the
following: if $p=1$, then $\Gamma\subset\Pi$ is arbitrary; if $p=2$, then
$\Gamma=\Pi\setminus\{\alpha_j\}$ for $1\leq j\leq n$; if $p\geq3$, then
$\Gamma$ is either $\Pi\setminus\{\alpha_1\}$ or $\Pi\setminus\{\alpha_n\}$.

Let us assume first $p\geq2$. If $\Gamma=\Pi\setminus\{\alpha_1\}$, then
$\Gamma^{\sf n}$ contains the real roots $e_1$ and $e_1\pm e_2$ so that the
minimal $G_0$-orbit cannot be a hypersurface. If $2\leq j\leq n$ and
$\Gamma=\Pi\setminus\{\alpha_j\}$, then $\Gamma^{\sf n}$ contains the real roots
$e_1$ and $e_2$ so that the minimal $G_0$-orbit cannot be a hypersurface.

Assume now that $p=1$. If $\Gamma$ does not contain $\alpha_j$ for some $1\leq
j\leq n-1$, then $\Gamma^{\sf n}$ contains $e_1\pm e_n$, so that the minimal
$G_0$-orbit cannot be a hypersurface. On the other hand, for $\Gamma=\Pi
\setminus\{\alpha_n\}$ we have $\Gamma^{\sf n}=\{e_1,\dotsc,e_n,e_k+e_l;\ 1\leq
k<l\leq n\}$, hence $\Gamma^{\sf n}\cap\sigma(\Gamma^{\sf n})=\{e_1\}$. In this
case the minimal $G_0$-orbit is a hypersurface.

\begin{rem}\label{Rem:IsotrGrassmann}
Let $\Gamma=\Pi\setminus\{\alpha_n\}$. Then $G_0={\rm{SO}}(1,2n)$ has a compact
hypersurface orbit in $X=G/Q_\Gamma$. We have $\dim_\mbb{C}X=\abs{\Gamma^{\sf
n}}=n(n+1)/2$. Note that $X\simeq{\rm{IGr}}_n(\mbb{C}^{2n+1})$.
\end{rem}

\subsection{The series ${\sf C}_n$}

Let $\lie{g}=\lie{sp}(n,\mbb{C})$ with $n\geq2$. The root system of $\lie{g}$
is given by
\begin{equation*}
\Delta=\{2e_k;\ 1\leq k\leq n\}\cup\{\pm e_k\pm e_l;\ 1\leq k<l\leq n\}
\end{equation*}
where $(e_1,\dotsc,e_n)$ is the standard basis of $\mbb{R}^n$. For $\Delta^+:=
\{2e_k, e_k\pm e_l;\ 1\leq k<l\leq n\}$ we have
\begin{equation*}
\Pi=\{\alpha_1=e_1-e_2,\alpha_2=e_2-e_3,\dotsc,\alpha_{n-1}=e_{n-1}-e_n,
\alpha_n=2e_n\}.
\end{equation*}
A direct calculation shows that
\begin{align*}
2e_k&=2(\alpha_k+\dotsb+\alpha_{n-1})+\alpha_n\\
e_k-e_l&=\alpha_k+\dotsb+\alpha_{l-1}\\
e_k+e_l&=\alpha_k+\dotsb+\alpha_{l-1}+2(\alpha_l+\dotsb+\alpha_{n-1})+\alpha_n
\end{align*}
for all $1\leq k<l\leq n$.

The non-compact real forms of $\lie{g}$ are $\lie{sp}(n,\mbb{R})$ and
$\lie{sp}(p,q)$ with $1\leq p\leq q$ and $p+q=n$. Since $\lie{sp}(n,\mbb{R})$ is
a split real form, it is sufficient to consider $\lie{g}_0:=\lie{sp}(p,q)$. The
real rank of $\lie{g}_0$ is $p$ and the restricted root system coincides with
$({\sf BC})_p$ for $p<q$ and ${\sf C}_p$ for $p=q$. The action of $\sigma$ on
$\Delta$ is given by
\begin{equation*}
\sigma(e_k)=
\begin{cases}
e_{k+1} &: \text{if $1\leq k\leq 2p$ is odd}\\
e_{k-1} &: \text{if $1\leq k\leq 2p$ is even}\\
-e_k &: 2p+1\leq k\leq n\\
\end{cases}.
\end{equation*}

According to~\cite[Table~2]{HONO} the only $\Gamma\subset\Pi$ such that the
minimal $G_0$-orbit in $X=G/Q_\Gamma$ might be a hypersurface are the
following: if $p=1$, then $\Gamma=\Pi\setminus\{\alpha_k,\alpha_l\}$ for all
$1\leq k\leq l\leq n$ (with $k=l$ allowed); if $p=2$, then
$\Gamma=\Pi\setminus\{\alpha_k\}$ for all $1\leq k\leq n$; if $p\geq3$, then
the only possibilities for $\Gamma$ are $\Pi\setminus\{\alpha_k\}$ for
$k=1,2,3,n$ or $\Pi\setminus\{\alpha_1,\alpha_2\}$.

Let $p$ be arbitrary. For $\Gamma=\Pi\setminus\{\alpha_1\}$ we have $\Gamma^{\sf
n}=
\{e_1\pm e_2,\dotsc,e_1\pm e_n, 2e_1\}$ and thus $\Gamma^{\sf
n}\cap\sigma(\Gamma^{\sf n})=
\{e_1+e_2\}$. Hence, $G_0={\rm{Sp}}(p,q)$ has a compact hypersurface orbit in
$X=G/Q_\Gamma\simeq\mbb{P}_{2n-1}$. Now suppose that $\Gamma$ does not contain
the root $\alpha_k$ for some $k\geq2$. Then $\Gamma^{\sf n}$ contains $2e_1$ and
$2e_2=\sigma(2e_1)$, so that the minimal $G_0$-orbit in $X$ is not a
hypersurface.

\subsection{The series ${\sf D}_n$}

Let $\lie{g}=\lie{so}(2n,\mbb{C})$ with $n\geq4$.\footnote{Recall that
$\lie{so}(6,\mbb{C})\simeq\lie{sl}(4,\mbb{C})$.} The root system of $\lie{g}$
is given by
\begin{equation*}
\Delta=\{\pm e_k\pm e_l;\ 1\leq k<l\leq n\}
\end{equation*}
where $(e_1,\dotsc,e_n)$ is the standard basis of $\mbb{R}^n$. For $\Delta^+:=
\{e_k\pm e_l;\ 1\leq k<l\leq n\}$ we have
\begin{equation*}
\Pi=\{\alpha_1=e_1-e_2,\alpha_2=e_2-e_3,\dotsc,\alpha_{n-1}=e_{n-1}-e_n,
\alpha_n=e_{n-1}+e_n\}.
\end{equation*}

\begin{rem}
There exists an automorphism of $\Pi$ that exchanges $\alpha_{n-1}$ and
$\alpha_n$. Consequently, there exists an outer automorphism of $G={\rm{SO}}
(2n,\mbb{C})$ that maps $Q_{\Pi\setminus\{\alpha_{n-1}\}}$ onto
$Q_{\Pi\setminus\{\alpha_n\}}$ although these parabolic groups are not
conjugate in $G$. In particular, the corresponding homogeneous rational
manifolds are isomorphic. As hermitian symmetric spaces they are isomorphic to
${\rm{SO}}(2n)/{\rm{U}}(n)$.
\end{rem}

A direct calculation shows that
\begin{align*}
e_k-e_l&=\alpha_k+\dotsb+\alpha_{l-1}\\
e_k+e_{n-1}&=\alpha_k+\dotsb+\alpha_n\text{ for all $1\leq k\leq n-2$}\\
e_k+e_n&=\alpha_k+\dotsb+\alpha_{n-2}+\alpha_n\text{ for all $1\leq k\leq
n-2$}\\
e_k+e_l&=\alpha_k+\dotsb+\alpha_{l-1}+2(\alpha_l+\dotsb+\alpha_{n-2})+
\alpha_{n-1}+\alpha_n\text{ for all $1\leq k<l\leq n-2$}.
\end{align*}

The non-compact real forms of $\lie{g}$ are $\lie{so}^*(2n)$ and $\lie{so}(p,q)$
with $1\leq p\leq q$ and $p+q=2n$.

Consider first $\lie{g}_0=\lie{so}^*(2n)$. The real rank of $\lie{g}_0$ is
$[n/2]$ and the restricted root system is $({\sf BC})_m$ for $n=2m+1$ and ${\sf
C}_m$ if $n=2m$.

We start with the case that $n=2m$ is even. The corresponding involution of
$\Delta$ is induced by
\begin{equation*}
\sigma(e_k)=
\begin{cases}
e_{k+1}&:\text{$1\leq k\leq n$ is odd}\\
e_{k-1}&:\text{$1\leq k\leq n$ is even}\\
\end{cases}.
\end{equation*}
One verifies directly that $\Delta^+$ is a $\sigma$-order.

For $\Gamma=\Pi\setminus\{\alpha_1\}$ we have $\Gamma^{\sf n}=\{e_1\pm e_2,
\dotsc,e_1\pm e_n\}$ and hence $\Gamma^{\sf n}\cap\sigma(\Gamma^{\sf n})=
\{e_1+e_2\}$. Consequently, the minimal $G_0$-orbit in $X=G/Q_\Gamma$ is a
hypersurface. If $\Gamma$ does not contain $\alpha_2$, then $\Gamma^{\sf n}$
contains $e_1-e_3$ and $e_2-e_4=\sigma(e_1-e_3)$, i.e., the minimal $G_0$-orbit
in $X=G/Q_\Gamma$ is not a hypersurface. If $n\geq 6$ and if $\Gamma$ does not
contain $\alpha_k$ for $3\leq k\leq n$, then $\Gamma^{\sf n}$ contains the two
real roots $e_1+e_2$ and $e_3+e_4$. On the other hand, for $n=4$ and $\Gamma=
\Pi\setminus\{\alpha_3\}$ we obtain $\Gamma^{\sf n}\cap\sigma(\Gamma^{\sf n})=
\{e_1+e_2\}$, hence the minimal orbit of ${\rm{SO}}^*(8)$ in $X=G/Q_\Gamma$ is
a hypersurface in this case. One checks directly that in the remaining cases
${\rm{SO}}^*(8)$ does not have a compact hypersurface orbit.

Suppose now that $n=2m+1\geq5$ is odd. In this case the involution of $\Delta$
is given by
\begin{equation*}
\sigma(e_k)=
\begin{cases}
e_{k+1}&:\text{$1\leq k\leq n-1$ is odd}\\
e_{k-1}&:\text{$1\leq k\leq n-1$ is even}\\
-e_k&:k=n\\
\end{cases}.
\end{equation*}
As above we see that for $\Gamma=\Pi\setminus\{\alpha_1\}$ we have $\Gamma^{\sf
n}\cap\sigma(\Gamma^{\sf n})=\{e_1+e_2\}$, while the minimal $G_0$-orbit in
$X=G/Q_\Gamma$ is not a hypersurface if $\Gamma$ does not contain $\alpha_k$
for $2\leq k\leq n$.

In summary, the only cases in which the minimal orbit of $G_0={\rm{SO}}^*(2n)$
in $X=G/Q_\Gamma$ is a hypersurface are $\Gamma=\Pi\setminus\{\alpha_1\}$ as
well as $n=4$ and $\Gamma=\Pi\setminus\{\alpha_3\}$.

\begin{rem}
The exceptional case $n=4$ is explained by $\lie{so}^*(8)\simeq\lie{so}(6,2)$
which corresponds to the fact that ${\rm{SO}}(8)/{\rm{U}}(4)$ is isomorphic to
the $3$-dimensional quadric.
\end{rem}

In the rest of this subsection we treat the case $\lie{g}_0=\lie{so}(p,q)$ with
$1\leq p\leq q$ and $p+q=2n$. The real rank of $\lie{g}_0$ is $p$ and the
restricted root system is ${\sf B}_p$ for $p<q$ and ${\sf D}_p$ for $p=q$.

\begin{rem}
The Lie algebra $\lie{so}(n,n)$ is a split real form of $\lie{g}$. The Lie
algebra $\lie{so}(1,2n-1)$ contains only one conjugacy class of Cartan
subalgebras, see~\cite[Appendix~C.3]{Kn}.
\end{rem}

The involution of $\Delta$ is induced by
\begin{equation*}
\sigma(e_k)=
\begin{cases}
e_k&:1\leq k\leq p\\
-e_k&:p+1\leq k\leq p+\left[\frac{q-p}{2}\right]=n\\
\end{cases}.
\end{equation*}

According to~\cite{HONO} the minimal $G_0$-orbit in $X=G/Q_\Gamma$ may be a
hypersurface only in the following cases. If $p=1$, then $\Gamma\subset\Pi$ is
arbitrary; if $p=2$, then $\Gamma$ coincides with $\Pi\setminus\{\alpha_k\}$ or
$\Pi\setminus\{\alpha_k,\alpha_{n-1}\}$ or $\Pi\setminus\{\alpha_k,\alpha_n\}$
for any $k$; if $p\geq3$, then the only possibilities for $\Gamma$ are
$\Pi\setminus\{\alpha_1\}$ or $\Pi\setminus\{\alpha_{n-1}\}$ or $\Pi\setminus
\{\alpha_n\}$.

Let us begin with the case $p\geq3$. If $\Gamma=\Pi\setminus\{\alpha_k\}$ for
$k=1,n-1,n$, then $\Gamma^{\sf n}$ contains the real roots $e_1+e_2$ and
$e_2+e_3$ so that the minimal $G_0$-orbit in $X=G/Q_\Gamma$ cannot be a
hypersurface.

Suppose now that $p=2$. If $\Gamma$ does not contain $\alpha_k$ for some $1\leq
k\leq n-2$, then $\Gamma^{\sf n}$ contains $e_1\pm e_n$. Since $\sigma(e_1-e_n)
=e_1+e_n$, the minimal $G_0$-orbit is not a hypersurface in this case. If
$\Gamma=\Pi\setminus\{\alpha_{n-1}\}$, then we have $\Gamma^{\sf n}=\bigr\{
e_1-e_n,\dotsc,e_{n-1}-e_n,e_k+e_l\ (1\leq k<l\leq n-1)\bigl\}$ and one
verifies $\Gamma^{\sf n}\cap\sigma(\Gamma^{\sf n})=\{e_1+e_2\}$. Hence, the
minimal $G_0$-orbit in $X=G/Q_\Gamma$ is a hypersurface. For $\Gamma=\Pi
\setminus\{\alpha_n\}$ we have $\Gamma^{\sf n}=\{e_k+e_l:\ 1\leq k<l\leq n\}$
and obtain again $\Gamma^{\sf n}\cap\sigma(\Gamma^{\sf n})=\{e_1+e_2\}$, which
leads to the same conclusion as above.

\begin{rem}
For $p=1$ the above considerations show that $G_0$ acts transitively on
$X=G/Q_\Gamma$ for $\Gamma=\Pi\setminus\{\alpha_k\}$ where $k=n-1,n$.
\end{rem}

In summary, the only cases in which the minimal orbit of $G_0={\rm{SO}}(p,q)$
in $X=G/Q_\Gamma$ is a hypersurface are $p=2$ and $\Gamma=\Pi\setminus
\{\alpha_k\}$ for $k=n-1,n$.

\subsection{The exceptional Lie algebra $\lie{g}=E_6$}

Combined with the general remarks in~\cite{Ara}, the Satake diagrams yield
explicit formulas of the involutions corresponding to the non-split non-compact
real forms of the exceptional Lie algebras $E_6$, $E_7$, $E_8$ and $F_4$.

Let $\lie{g}=E_6$. Identifying $\lie{h}_\mbb{R}^*$ with $V=\{x\in\mbb{R}^8;\
x_6=x_7=-x_8\}$ a system of simple roots is given by
\begin{equation*}
\Pi=\bigl\{\alpha_1=1/2(e_1-e_2-e_3-e_4-e_5-e_6-e_7+e_8, \alpha_2=e_1+e_2,
\alpha_j=e_{j-1}-e_{j-2} (3\leq j\leq6)\bigr\}.
\end{equation*}

The Lie algebra $\lie{g}=E_6$ has two non-split non-compact real forms, namely
$EII$ and $EIII$.

Suppose first that $\lie{g}_0=EII$. Since there is no imaginary simple root,
the Satake diagram of $\lie{g}_0$ determines directly the involution
$\sigma\colon\Delta^+\to\Delta^+$. More precisely, we have
\begin{equation*}
\sigma(\alpha_1)=\alpha_6, \sigma(\alpha_3)=\alpha_5,
\sigma(\alpha_2)=\alpha_2, \sigma(\alpha_4)=\alpha_4.
\end{equation*}
According to~\cite{HONO} we must only check $\Pi\setminus\{\alpha_1\}$ and
$\Pi\setminus\{\alpha_6\}$.

Let $\Gamma=\Pi\setminus\{\alpha_j\}$ for $j=1,6$. In both cases $\Gamma^{\sf
n}\cap\sigma (\Gamma^{\sf n})$ contains the two real roots
$\alpha_1+\alpha_3+\alpha_4+ \alpha_5+\alpha_6$ and
$\alpha_1+\alpha_2+\alpha_3+\alpha_4+\alpha_5+ \alpha_6$. Hence the minimal
$G_0$-orbit in $X=G/Q_\Gamma$ is not a hypersurface.

Suppose now that $\lie{g}_0=EIII$. It can be seen from its Satake diagram that
$\Pi_{\sf i}=\{\alpha_3,\alpha_4,\alpha_5\}$ and that
\begin{align*}
\sigma(\alpha_1)&=\alpha_6+C_{1,3}\alpha_3+C_{1,4}\alpha_4+C_{1,5}\alpha_5\\
\sigma(\alpha_2)&=\alpha_2+C_{2,3}\alpha_3+C_{2,4}\alpha_4+C_{2,5}\alpha_5\\
\sigma(\alpha_6)&=\alpha_1+C_{6,3}\alpha_3+C_{6,4}\alpha_4+C_{6,5}\alpha_5.
\end{align*}
Since $\sigma$ is involutive, we obtain $C_{6,j}=C_{1,j}$ for $j=3,4,5$. This
gives
\begin{equation*}
\sigma(\alpha_1+\alpha_3+\alpha_4+\alpha_5+\alpha_6)=\alpha_1+(2C_{1,3}-1)
\alpha_3+(2C_{1,4}-1)\alpha_4+(2C_{1,5}-1)\alpha_5+\alpha_6.
\end{equation*}
Comparison with the list of positive roots shows that $\alpha_1+\alpha_3+
\alpha_4+\alpha_5+\alpha_6$ must be a real root, i.e., that $C_{1,3}=C_{1,4}=
C_{1,5}=1$. Similarly, the only possibilities for $\sigma(\alpha_2)$ are
$\alpha_2$, $\alpha_2+\alpha_4$, $\alpha_2+\alpha_4+\alpha_5$,
$\alpha_2+\alpha_3+\alpha_4$, $\alpha_2+\alpha_3+\alpha_4+\alpha_5$ and
$\alpha_2+\alpha_3+2\alpha_4+\alpha_5$. However, since we know that
$\sigma(\alpha_2)-\alpha_2$ is not a root, we only have $\sigma(\alpha_2)=
\alpha_2$ or $\sigma(\alpha_2)=\alpha_2+\alpha_3+2\alpha_4+\alpha_5$. In the
first case we obtain $\sigma(\alpha_2+\alpha_4)=\alpha_2-\alpha_4$, which
contradicts the fact that $\Delta^+$ is a $\sigma$-order. Therefore, we see
that $\sigma(\alpha_2)=\alpha_2+\alpha_3+2\alpha_4+\alpha_5$.

Let $\Gamma=\Pi\setminus\{\alpha_j\}$ for $1\leq j\leq6$. Then $\Gamma^{\sf n}
\cap\sigma(\Gamma^{\sf n})$ contains always the roots $\alpha_1+\alpha_2+
\alpha_3+\alpha_4+\alpha_5+\alpha_6$ and
\begin{equation*}
\sigma(\alpha_1+\alpha_2+\alpha_3+\alpha_4+\alpha_5+\alpha_6)=
\alpha_1+\alpha_2+2\alpha_3+3\alpha_4+2\alpha_5+\alpha_6.
\end{equation*}
Consequently, the minimal $G_0$-orbit in $X=G/Q_\Gamma$ cannot be a
hypersurface for any $\Gamma\subset\Pi$.

\subsection{The exceptional Lie algebra $\lie{g}=E_7$}

Let $\lie{g}=E_7$. Identifying $\lie{h}_\mbb{R}^*$ with $V=\{x\in\mbb{R}^8;\
x_8=-x_7\}$ a system of simple roots is given by
\begin{equation*}
\Pi=\bigl\{\alpha_1=1/2(e_1-e_2-e_3-e_4-e_5-e_6-e_7+e_8), \alpha_2=e_1+e_2,
\alpha_j=e_{j-1}-e_{j-2} (3\leq j\leq7)\bigr\}.
\end{equation*}
The Lie algebra $\lie{g}=E_7$ has two non-split non-compact real forms, namely
$EVI$ and $EVII$.

Let $\lie{g}_0=EVI$. Its Satake diagram shows $\Pi_{\sf i}=\{\alpha_2,
\alpha_5, \alpha_7\}$. In a first step we determine the integers $C_{k,l}$
such that
\begin{equation*}
\sigma(\alpha_k)=\alpha_k+C_{k,2}\alpha_2+C_{k,5}\alpha_5+C_{k,7}\alpha_7
\end{equation*}
for $k=1,3,4,6$. One checks immediately that $\sigma(\alpha_1)=\alpha_1$ and
$\sigma(\alpha_3)=\alpha_3$. For the remaining cases the only possibilities
that respect $\sigma(\alpha_k)-\alpha_k\notin\Delta$ are
\begin{equation*}
\sigma(\alpha_4)=\alpha_4\text{ or }\sigma(\alpha_4)=\alpha_2+\alpha_4+\alpha_5
\end{equation*}
and
\begin{equation*}
\sigma(\alpha_6)=\alpha_6\text{ or }\sigma(\alpha_6)=\alpha_5+\alpha_6+\alpha_7.
\end{equation*}
Since $\alpha_4+\alpha_5,\alpha_5+\alpha_6\in\Delta^+$ we obtain
\begin{align*}
\sigma(\alpha_4)&=\alpha_2+\alpha_4+\alpha_5\text{ and}\\
\sigma(\alpha_6)&=\alpha_5+\alpha_6+\alpha_7.
\end{align*}
According to~\cite{HONO} the only possibility for a minimal orbit of
hypersurface type is $\Gamma=\Pi\setminus\{\alpha_7\}$. Since in this case
$\Gamma^{\sf n}$ contains the two real roots $\alpha_2+\alpha_3+2\alpha_4+
2\alpha_5+2\alpha_6+\alpha_7$ and $\alpha_1+\alpha_2+\alpha_3+2\alpha_4+
2\alpha_5+2\alpha_6+\alpha_7$, the minimal $G_0$-orbit in $X=G/Q_\Gamma$ cannot
be a hypersurface.

Let $\lie{g}_0=EVII$. Here we have $\Pi_{\sf i}=\{\alpha_2,\alpha_3,\alpha_4,
\alpha_5\}$ and we must determine
\begin{equation*}
\sigma(\alpha_k)=\alpha_k+C_{k,2}\alpha_2+C_{k,3}\alpha_3+C_{k,4}\alpha_4+
C_{k,5}\alpha_5
\end{equation*}
for $k=1,6,7$. One sees directly $\sigma(\alpha_7)=\alpha_7$. For the remaining
cases the only possibilities that respect $\sigma(\alpha_k)-\alpha_k
\notin\Delta$ are
\begin{equation*}
\sigma(\alpha_1)=\alpha_1\text{ or }\sigma(\alpha_1)=\alpha_1+\alpha_2+2
\alpha_3+2\alpha_4+\alpha_5
\end{equation*}
and
\begin{equation*}
\sigma(\alpha_6)=\alpha_6\text{ or }\sigma(\alpha_6)=\alpha_2+\alpha_3+
2\alpha_4+2\alpha_5+\alpha_6.
\end{equation*}
Since $\alpha_1+\alpha_3,\alpha_5+\alpha_6\in\Delta^+$ we obtain
\begin{align*}
\sigma(\alpha_4)&=\alpha_1+\alpha_2+2\alpha_3+2\alpha_4+\alpha_5\text{ and}\\
\sigma(\alpha_6)&=\alpha_2+\alpha_3+2\alpha_4+2\alpha_5+\alpha_6.
\end{align*}

Let $\Gamma=\Pi\setminus\{\alpha_k\}$ for $1\leq k\leq 7$. Then $\Gamma^{\sf n}
\cap\sigma(\Gamma^{\sf n})$ contains always $\alpha_1+\alpha_2+\alpha_3+
\alpha_4+\alpha_5+\alpha_6+\alpha_7$ and
\begin{equation*}
\sigma(\alpha_1+\alpha_2+\alpha_3+\alpha_4+\alpha_5+\alpha_6+\alpha_7)=
\alpha_1+\alpha_2+2\alpha_3+3\alpha_4+2\alpha_5+\alpha_6+\alpha_7.
\end{equation*}
Consequently, the minimal $G_0$-orbit in $X=G/Q_\Gamma$ is never a hypersurface.

\subsection{The exceptional Lie algebra $\lie{g}=E_8$}

According to~\cite{HONO} no real form of $G$ can have a compact hypersurface in
any $G$-homogeneous rational manifold.

\subsection{The exceptional Lie algebra $\lie{g}=F_4$}

The rank of $\lie{g}=F_4$ is $4$ and the root system is given by
\begin{equation*}
\Delta=\{e_k;\ 1\leq k\leq4\}\cup\{\pm e_k\pm e_l;\ 1\leq k<l\leq4\}\cup
\bigl\{1/2(\pm e_1\pm e_2\pm e_3\pm e_4)\bigr\}.
\end{equation*}
Choosing $\Delta^+=\{e_k\}\cup\{e_k\pm e_l\}\cup\{1/2(e_1\pm e_2\pm e_3\pm
e_4)\}$ we obtain
\begin{equation*}
\Pi=\{\alpha_1=1/2(e_1-e_2-e_3-e_4), \alpha_2=e_4, \alpha_3=e_3-e_4,
\alpha_4=e_2-e_3\}.
\end{equation*}
The non-compact real forms of $\lie{g}$ are $FI$ and $FII$. Since $FI$ is
split, we concentrate on $\lie{g}_0=FII$. According to~\cite[p.~21]{Ara} the
simple roots $\alpha_2$, $\alpha_3$ and $\alpha_4$ are imaginary while
$\sigma(\alpha_1)=\alpha_1+3\alpha_2+2\alpha_3+\alpha_4$. Equivalently, we
have $\sigma(e_1)=e_1$ and $\sigma(e_k)=-e_k$ for $2\leq k\leq 4$. One checks
that $\Delta^+$ is a $\sigma$-order.

A direct calculation shows that the minimal $G_0$-orbit in $X=G/Q_\Gamma$ is
never a hypersurface.

\subsection{The exceptional Lie algebra $\lie{g}=G_2$}

The only non-compact real form of $\lie{g}$ is split.

\end{appendix}

\end{document}